\documentclass[11pt]{article}
\usepackage{amsmath, amscd, amssymb, latexsym, epsfig, color, amsthm,tikz, url}
\usepackage[all]{xy}
\usepackage{verbatim}
\usetikzlibrary{arrows,shapes,positioning,calc,decorations.pathreplacing}
\numberwithin{equation}{section}
\def\proof{\smallskip\noindent {\it Proof: \ }}
\def\endproof{\hfill$\square$\medskip}
\newtheorem{theorem}{Theorem}[section]
\newtheorem{proposition}[theorem]{Proposition}
\newtheorem{question}[theorem]{Question}
\newtheorem{corollary}[theorem]{Corollary}
\newtheorem{conjecture}[theorem]{Conjecture}
\newtheorem{lemma}[theorem]{Lemma}

\newtheorem{problem}[theorem]{Problem}

\theoremstyle{definition}
\newtheorem{definition}[theorem]{Definition}
\newtheorem{example}[theorem]{Example}

\newtheorem{remark}[theorem]{Remark}

\DeclareMathOperator{\skel}{Skel}
\DeclareMathOperator\lk{\mathrm{lk}}

\DeclareMathOperator\st{\mathrm{st}}
\DeclareMathOperator{\intr}{\mathrm{int}}

\DeclareMathOperator{\conv}{\mathrm{conv}}

\newcommand{\relint}{\mathrm{relint}}

\newcommand{\R}{{\mathbb R}}

\newcommand{\Sp}{\mathbb{S}}

\title{Missing faces of neighborly and nearly neighborly polytopes and spheres}
\author{
	Isabella Novik\thanks{Research of IN is partially\textsl{} supported by NSF grant  DMS-2246399.}\\
	\small Department of Mathematics\\[-0.8ex]
	\small University of Washington\\[-0.8ex]
	\small Seattle, WA 98195-4350, USA\\[-0.8ex]
	\small \texttt{novik@uw.edu}
	\and 
	Hailun Zheng\thanks{Research of HZ is partially\textsl{} supported by NSF grant DMS-2246793.} \\
	\small Department of Mathematics\\[-0.8ex]
	\small University of Hawai`i at M\={a}noa\\[-0.8ex]
	\small 2565 McCarthy Mall, Honolulu, HI 96822, USA \\[-0.8ex]
	\small \texttt{hailunz@hawaii.edu}
}			
\begin{document}	
	\maketitle
	\begin{abstract}
		For a $(d-1)$-dimensional simplicial complex $\Delta$ and  $1\leq i\leq d$, let $f_{i-1}$ be the number of $(i-1)$-faces of $\Delta$ and $m_i$ be the number of missing $i$-faces of $\Delta$. In the nineties, Kalai asked for a characterization of the $m$-numbers of simplicial polytopes and spheres --- a problem that remains wide open to this day. Here, we study the $m$-numbers of nearly neighborly and neighborly polytopes and spheres.  Specifically, for $d\geq 4$, we obtain a lower bound on $m_{\lfloor d/2\rfloor}$ in terms of $f_0$ and $f_{\lfloor d/2\rfloor-1}$ in the class of all $(\lfloor d/2\rfloor-1)$-neighborly $(d-1)$-spheres. For neighborly spheres, 	we (almost) characterize the $m$-numbers of 
		$2$-neighborly $4$-spheres, and we show that, for all odd values of $k$, there exists an infinite family of $k$-neighborly simplicial $2k$-spheres with $m_{k+1}=0$. Along the way, we provide a simple numerical condition based on the $m$-numbers that allows to establish non-polytopality of some neighborly odd-dimensional spheres.
	\end{abstract}
	\section{Introduction}
	A missing face $F$ of a simplicial complex $\Delta$ is a subset of the vertex set of $\Delta$ that is not a face but such that all proper subsets of $F$ are faces.  The missing faces of $\Delta$ correspond to the minimal generators of the Stanley--Reisner ideal of $\Delta$. In other words, the collection of the missing faces, together with the vertex set, contains the same information as the collection of faces. Yet, while the face numbers of simplicial polytopes and spheres are completely characterized by the celebrated $g$-theorem (see \cite{BilleraLee,Stanley80} and \cite{Adiprasito-g-conjecture,AdiprasitoPapadakisPetrotou, KaruXiao,PapadakisPetrotou}), much less is known at present about the numbers of missing faces of polytopes and spheres. Specifically, the following problem of Gil Kalai (see Problem 19.5.42 in \cite{Kalai-skeletons}) remains wide open. For a $(d-1)$-dimensional simplicial complex or a simplicial $d$-polytope $\Delta$, let $f_i=f_i(\Delta)$ be the number of $i$-faces of $\Delta$ and let $m_i=m_i(\Delta)$ be the number of missing $i$-faces of $\Delta$. Define the $f$-vector and the $m$-vector of $\Delta$ as  $$f(\Delta)=(f_{-1},f_0,\dots,f_{d-1}) \quad \text{and} \quad m(\Delta)=(m_1, m_2, \ldots, m_d),\, \text{respectively}.$$
	
	\begin{problem} \label{Problem:Kalai}
		Characterize the $m$-vectors of simplicial $d$-polytopes and the $m$-vectors of simplicial $(d-1)$-spheres in terms of their $f$-vectors. 
	\end{problem}
	
	Tight upper bounds on the $m$-numbers of simplicial $d$-polytopes and simplicial $(d-1)$-spheres in terms of their $f$-numbers (equivalently, their $g$-numbers) were established by Nagel \cite{Nagel}. Nagel proved that the $m$-numbers are maximized by the Billera--Lee polytopes, thus settling a conjecture proposed by Kalai, Kleinschmidt, and Lee \cite[Conjecture 2]{Kalai-survey}. On the other hand, the lower bounds on the numbers of missing faces of spheres and polytopes remain very mysterious.
	The goal of this paper is to start developing such bounds. 
	
	Finding lower bounds on the $m$-numbers is ultimately related to several long-standing problems in extremal combinatorics. For instance, the clique density problem asks what is the minimum number of $r$-cliques in a graph with $f_0$ vertices and $f_1$ edges; see \cite{LPS, Nikoforov, Razborov,Reiher} for spectacular recent advances on this problem. Since for a simplicial complex $\Delta$, the number of $3$-cliques in the graph of $\Delta$ is equal to $f_2+m_2$, any lower bound on the number of $3$-cliques, such as for instance  Goodman's bound \cite{Goodman, MoonMoser,NordhausStewart}, gives a lower bound on $m_2$ in terms of $f_0,f_1,f_2$. However, many graphs cannot be realized as graphs of simplicial $(d-1)$-spheres, and so, even a tight lower bound on the number of $3$-cliques does not necessarily yield a tight lower bound on $m_2$ in Kalai's question. Consequently, the problem of establishing tight lower bounds on the $m$-numbers in its full generality is rather unmanageable at present. 
	
	In this paper, we mostly concentrate on the classes of 
	simplicial $d$-polytopes and $(d-1)$-spheres that are $\lfloor d/2\rfloor$- or $(\lfloor d/2\rfloor-1)$-neighborly. We refer to the former polytopes and spheres as neighborly and to the latter as nearly neighborly. (For instance, every simplicial polytope of dimension $4$ or $5$ and every simplicial sphere of dimension $3$ or $4$ are nearly neighborly.) In each of these two classes, almost all of the $m$-numbers are fixed functions of $f_0$ and $d$. 
	Specifically, every $k$-neighborly $(2k-1)$-sphere with $n\geq 2k+2$ vertices (i.e., not the boundary of a simplex) has $m_i=0$ for $i\neq k$
	and $m_k={n-k-1\choose k+1}+{n-k-2 \choose k}$. 
	(For $i<k$, the result about zeros follows from neighborliness, and for $i>k$, it is a consequence of the Alexander duality.)  
	Similarly, a $k$-neighborly $2k$-sphere with $n\geq 2k+3$ vertices has $m_i=0$, for all $i\neq k, k+1$, and $m_k={n-k-2 \choose k+1}$. On the other hand, $m_{k+1}$ could vary, and the only currently known condition is $m_{k+1}\leq {n-k-3 \choose k}$, with the upper bound achieved by the boundary complex of the cyclic polytope. Are there $k$-neighborly $2k$-spheres with $m_{k+1}$ equal to zero? More generally, what integers between $0$ and ${n-k-3 \choose k}$ can be realized as $m_{k+1}$ of some $k$-neighborly $2k$-sphere with $n$ vertices? How do the $m$-numbers of nearly neighborly spheres behave? More precisely, in the class of $(k-1)$-neighborly $(d-1)$-spheres with $n$ vertices, where $d\in\{2k,2k+1\}$, is $m_k$ bounded from below by some function of $n$ and $f_{k-1}$? These are the types of questions we address in the paper.

	Our main results can be summarized as follows.
	\begin{itemize}
		\item Let $k\geq 2$. Extending Goodman's bound on $m_2$,  we derive a lower bound on $m_k$ in terms of $n$, $f_{k-1}$, and $f_k$ in the  class of all simplicial complexes of dimension $\geq k-1$ with $n$ vertices; see Theorem \ref{lm: k-1-neighborly higher dimensional spheres}.

		\item As a corollary of the above result, for $k\geq 2$ and $d\in\{2k,2k+1\}$, we establish a lower bound on $m_k$ in terms of $n$ and $f_{k-1}$ in the class of all $(d-1)$-dimensional $(k-1)$-neighborly Eulerian complexes with $n$ vertices; see Corollary \ref{lm: near neighborly spheres}. This provides a step toward a resolution of Problem \ref{Problem:Kalai} for $3$-spheres and $4$-spheres 
		(Section 5.2) as well as a new upper bound on the number of edges of $4$-dimensional flag Eulerian complexes with $n$ vertices (Corollary \ref{cor:flag}).
		
		\item We (almost) characterize the $m$-vectors of $2$-neighborly $4$-spheres with $n$ vertices. Specifically, we show that for all $n\geq 9$ and $0\leq m\leq {n-5 \choose 2}$, except possibly for $m={n-5 \choose 2}-1$, there exists a $2$-neighborly $4$-sphere with $n$ vertices and $m_3=m$; see Theorem \ref{thm: neighborly 4-spheres}. 
		
		\item We also prove that for every odd $k$ and $n\geq 2k+4$, there exists a $k$-neighborly simplicial $(2k+1)$-polytope with $n$ vertices and $m_{k+1}=0$; see Theorem \ref{thm: neighborly 2k-sphere m_k+1=0}.
		
		\item Along the way, we show that if $P$ is a $(k+1)$-neighborly $(2k+2)$-polytope with $n+1$ vertices, then all vertex links of $\partial P$ must have $m_{k+1}={n-k-3 \choose k}$. Consequently, if a neighborly $(2k+1)$-sphere has a vertex link that violates this equality, then the sphere is not the boundary complex of any polytope; see Theorem \ref{non-polytopal}.
	\end{itemize}
	
	Testing polytopality is a hard problem that received a lot of attention in the recent years \cite{GMW,Pfeifle-20}. Our numerical condition in the last bullet point is motivated by Gr\"unbaum--Sreedharan's non-polytopal $3$-sphere with eight vertices \cite{GruSree} as well as by works of Perles (unpublished) and Bagchi and Datta \cite{BD2}. Example \ref{ex:non-polytopality} provides a list of several $3$- and $5$-dimensional vertex-transitive neighborly spheres whose non-polytopality is an immediate consequence of this numerical condition.
	
	The proofs of our main theorems rely on such results and techniques as the Dehn--Sommerville relations, characterizations of $k$-stacked spheres, Pachner’s bistellar flips, Shemer’s sewing operations, and Gale diagrams. Throughout the paper, we also discuss many open problems.

	The rest of the paper is structured as follows. In Section 2, we review basics of simplicial polytopes and spheres. In Section 3, we summarize previous results on the upper bounds of the $m$-numbers of simplicial spheres and derive simple corollaries. As an application, we show in Section 4 that numerical conditions on the $m$-numbers can be used to establish non-polytopality of some neighborly odd-dimensional spheres. In Section 5, we prove a generalization of Goodman's bound and use this result to provide lower bounds on the $m$-numbers of nearly neighborly Eulerian complexes. Section 6 is devoted to the $m$-numbers of $2$-neighborly $4$-spheres. We end in Section 7 by proving that for 
	$k=2$ as well as for every odd $k\geq 3$, there exists an infinite family of $k$-neighborly simplicial $(2k+1)$-polytopes with $m_{k+1}=0$. 
	
	\section{Review of simplicial polytopes and spheres}
	In this section we review results and definitions related to simplicial polytopes and spheres that will be used throughout the paper.
	
	A {\em polytope} $P \subseteq \R^d$ is the convex hull of a finite set of points in $\R^d$. The {\em dimension} of $P$ is the dimension of the affine span of $P$ and we say that $P$ is a {\em $d$-polytope} if $P$ is $d$-dimensional. An important example of a $d$-polytope is the {\em $d$-simplex}. It is defined as the convex hull of $d+1$ affinely independent points and is denoted $\sigma^d$.

	A {\em  (proper) face of a polytope $P$} is the intersection of $P$ with any supporting hyperplane of $P$. The {\em dimension} of a face $F$ of $P$ is the dimension of the affine span of $F$, and we say that $F$ is an {\em $i$-face} if $\dim F=i$. For a vertex $v$ of $P$, a {\em vertex figure of $P$ at $v$}, $P/v$, is the polytope obtained by intersecting $P$ with a hyperplane that separates $v$ from all other vertices of $P$. A polytope $P$ is called {\em simplicial} if all of its (proper) faces are simplices. 
	
	An (abstract) {\em simplicial complex} $\Delta$ with vertex set $V=V(\Delta)$ is a {non-empty} collection of subsets of $V$ that is closed under inclusion and contains all singletons: $\{v\}\in\Delta$ for all $v\in V$. An example of a simplicial complex on $V$ is the collection of all subsets of $V$, denoted $\overline{V}$, and called the (abstract) {\em simplex} on $V$.
	
	The elements of a simplicial complex $\Delta$ are called {\em faces} of $\Delta$. A face $F$ of $\Delta$ has {\em dimension} $i$ if $|F|=i+1$; in this case we say that $F$ is an {\em $i$-face}. We usually refer to $0$-faces as {\em vertices}, $1$-faces as {\em edges}, and the maximal under inclusion faces as {\em facets}. For brevity we denote a vertex by $v$ and an edge by $uv$ instead of $\{v\}$ and $\{u, v\}$ respectively. The {\em dimension of $\Delta$} is $\max\{\dim F: F\in \Delta\}$. For instance, the dimension of $\overline{V}$ is $|V|-1$.
	
	A set $F\subseteq V$ is a {\em missing face} of $\Delta$ if $F$ is not a face of $\Delta$, but every proper subset of $F$ is a face of $\Delta$. In analogy with faces, a {\em missing $i$-face} is a missing face of size $i+1$.
	The collection of the missing faces of $\Delta$, together with the vertex set of $\Delta$, uniquely determines $\Delta$. A complex $\Delta$ is {\em flag} if all missing faces of $\Delta$ are $1$-dimensional.
	
	There are several subcomplexes of a simplicial complex $\Delta$ that will be useful. If $F$ is a face of $\Delta$, then the {\em star of $F$} and the  {\em link of $F$} are 
	$$\st(F)=\st(F,\Delta)=\{\sigma \in \Delta \ : \  \sigma\cup F\in\Delta\} \;\;\text{ and }\;\;\lk(F)=\lk(F,\Delta)= \{\sigma\in \st(F) \ : \ \sigma\cap F=\emptyset\}.$$ The subcomplex of $\Delta$ consisting of all faces of $\Delta$ of dimension $\leq k$ is called the {\em $k$-skeleton} of $\Delta$ and is denoted $\skel_k(\Delta)$; the $1$-skeleton of $\Delta$ is also known as the {\em graph} of $\Delta$.
	A subcomplex of $\Delta$ is called {\em induced} if it is of the form $\Delta[W]=\{F\in \Delta: F\subseteq W\}$ for some $W\subseteq V(\Delta)$. Finally, if $\Delta$ and $\Gamma$ are two simplicial complexes  on disjoint vertex sets, then the {\em join} of $\Delta$ and $\Gamma$ is 
	$$\Delta*\Gamma=\{\sigma\cup \tau: \sigma\in \Delta, \tau \in \Gamma\}.$$ For brevity, we denote the cone over a complex $\Gamma$ with apex $v$ by $\Gamma*v$.

	A simplicial complex $\Delta$ is called a {\em simplicial $d$-ball} ({\em simplicial $(d-1)$-sphere}, resp.) if its geometric realization is homeomorphic to a $d$-dimensional ball ($(d-1)$-dimensional sphere, resp.). Occasionally, we will also work with {\em piecewise linear} balls and spheres ({\em PL} for short) as well as with {\em $\R$-homology balls}, which we will simply call {\em homology balls}. All PL balls (spheres, resp.) are simplicial balls (spheres, resp), and all simplicial balls are  homology balls.  
	A face $F$ of a (PL, simplicial, or homology) ball $B$ is called a {\em boundary face} if the link of $F$ has vanishing homology over $\R$, and it is an {\em interior face} otherwise. A {\em minimal interior face} of $B$ is an interior face that contains no other interior faces. The set of all boundary faces of $B$ forms a simplicial complex called the {\em boundary complex} of $B$; it is denoted $\partial B$.
	
	If $P$ is a simplicial polytope, then the {\em boundary complex} of $P$, denoted $\partial P$, consists of the collection of vertex sets of (proper) faces of $P$. The boundary complex of $P$  is a simplicial complex. We refer to missing faces of $\partial P$ as missing faces of $P$. We also say that two simplicial polytopes $P$ and $Q$ are {\em combinatorially equivalent} if $\partial P$ and $\partial Q$ are isomorphic simplicial complexes.
	
	The class of boundary complexes of simplicial $d$-polytopes is contained in the class of simplicial $(d-1)$-spheres. By Steinitz' theorem, these two classes are equal when $d=3$; however, the inclusion is strict for all $d>3$, see \cite{GoodmanPollack, Kal, PfeiZieg}. We say that a simplicial sphere $\Delta$ is {\em polytopal} if it is the boundary complex of a simplicial polytope. The links of polytopal spheres are polytopal. In fact, the boundary complex of $P/v$ is the link  $v$ in $\partial P$.
	
	In what follows, let $\Delta$ be either a simplicial $d$-polytope or a $(d-1)$-dimensional simplicial complex. Denote by $f_i=f_i(\Delta)$ the number of $i$-faces of $\Delta$ and by $m_i=m_i(\Delta)$ the number of missing $i$-faces of $\Delta$. In particular, $f_i=0$ if $i>d-1$. Let $$f(\Delta)=(f_{-1}, f_0, \dots, f_{d-1}) \quad\mbox{and}\quad m(\Delta)=(m_1, m_2, \dots, m_d)$$
	be the {\em $f$-vector} and the {\em $m$-vector} of $\Delta$, respectively. The {\em $h$-vector} of $\Delta$, $h(\Delta)=(h_0, h_1,\ldots,h_d)$, is obtained from the $f$-vector by the following invertible linear transformation:
	$$h_j= h_j(\Delta)=\sum_{i=0}^{j} (-1)^{j-i}\binom{d-i}{d-j}f_{i-1}(\Delta) \quad \text{for }\, 0\leq j\leq d.$$ 
	The {\em $g$-vector} of $\Delta$, $g(\Delta)=(g_0,g_1,\ldots, g_{\lfloor d/2\rfloor})$ is then defined by letting $g_0=1$ and $g_j=h_j-h_{j-1}$ for $1\leq j\leq \lfloor d/2\rfloor$.
	
	When $\Delta$ is a simplicial $(d-1)$-sphere, the $h$-numbers of $\Delta$ satisfy the Dehn--Sommerville relations: $h_i=h_{d-i}$ for all $0\leq i\leq d$ (see \cite{Klee64}). Hence in this case, the $g$-vector of $\Delta$ completely determines the $f$-vector of $\Delta$. When $d$ is odd, we will also sometimes consider $g_{(d+1)/2}:=h_{(d+1)/2}-h_{(d-1)/2}=0$.
	
	We say that $\Delta$ is {\em $i$-neighborly} if $f_{i-1}(\Delta)=\binom{f_0(\Delta)}{i}$. (This notion is only interesting when $i\geq 2$ as any simplicial complex is $1$-neighborly.) Simplicial $d$-polytopes and simplicial $(d-1)$-spheres with at least $d+2$ vertices can be at most $\lfloor d/2 \rfloor$-neighborly; in the case they are $\lfloor d/2\rfloor$-neighborly, we simply call them {\em neighborly}. Neighborly polytopes and spheres abound in nature; see \cite{NZ-neighborly, Padrol-13, Shemer}. The Upper Bound Theorem \cite{McMullen70, Stanley75} asserts that among all simplicial spheres of a fixed dimension and with a fixed number of vertices, the neighborly spheres simultaneously maximize all the face numbers.
	
	One famous example of neighborly polytopes is given by the family of cyclic polytopes.
	A {\em cyclic $d$-polytope} on $n$ vertices, denoted $C(d, n)$, is defined as the convex hull of $n>d$ distinct points on the moment curve $M(t)=\{(t, t^2, \dots, t^d): t\in \R\}$. The facets of a cyclic polytope are characterized by the Gale evenness condition. To state this condition, we let the vertices of $C(d,n)$ be $v_i=M(t_i)$, where $1\leq i\leq n$ and $t_1<t_2<\dots <t_n$. Then for a $d$-subset $I=\{i_1<\dots<i_d\}$ of $[n]:=\{1,2,\ldots,n\}$, the set $F_I=\conv(v_i : i\in I)$ is a facet of $C(d,n)$ if and only if any two elements of $[n]\backslash I$ are separated by an even number of elements from $I$; see \cite[Theorem 0.7]{Ziegler}. In particular, the combinatorial type of $C(d,n)$ is independent of the choice of $t_1,\ldots,t_n$, and so from now on we will refer to $C(d,n)$ as {\em the} cyclic polytope.

	While neighborly spheres maximize the face numbers, stacked spheres ---a notion we are about to define--- minimize the face numbers.
	A homology $d$-ball $\Delta$ is {\em $i$-stacked} if it has no interior faces of dimension $\leq d-i-1$. A simplicial $(d-1)$-sphere is {\em $i$-stacked} if it is the boundary complex of an $i$-stacked homology $d$-ball.  For instance, the $d$-simplex  is the only $0$-stacked $d$-ball, and its boundary complex is the only  $0$-stacked $(d-1)$-sphere. Any $1$-stacked $(d-1)$-sphere is polytopal and can be represented as the connected sum of the boundary complexes of $d$-simplices; $1$-stacked spheres are also called {\em stacked} spheres. 
	
	There are several numerical characterizations of $i$-stacked spheres. 
	First, when $0\leq i\leq \lfloor d/2 \rfloor -1$,  a simplicial $(d-1)$-sphere is $i$-stacked if and only if $g_{i+1}=0$. This criterion is a part of the Generalized Lower Bound Theorem; see \cite{Kalai87} for the case of $i=1$ and \cite[Theorem 1.3]{MuraiNevo2013} for the general case. Second, when $0\leq i\leq \lfloor d/2 \rfloor -1$, a simplicial $(d-1)$-sphere is $i$-stacked if and only if $m_{d-i}=g_i$. Furthermore, if $d=2i+1$ is odd and a simplicial $(d-1)$-sphere is $i$-stacked, then $m_{i+1}=g_i$;  see \cite[Corollary 1.4]{MNZ}.

	We close this section with a discussion of the $g$-theorem, which provides a complete characterization of the $f$-vectors of simplicial spheres. For the case of simplicial polytopes, this result was established in the eighties; see \cite{BilleraLee, Stanley80}. The proof for simplicial spheres is much more recent; see \cite{Adiprasito-g-conjecture, AdiprasitoPapadakisPetrotou, KaruXiao, PapadakisPetrotou}. The statement relies on the function $m^{\langle k \rangle}$ defined as follows. If $m$ and $k$ are positive integers, then there is a unique expression of $m$ in the form 
	$$m=\binom{a_k}{k}+\binom{a_{k-1}}{k-1}+\dots +\binom{a_i}{i}, \quad \mbox{where } a_k>a_{k-1}>\dots >a_i\geq i>0.$$
	Using this expression, we define
	$$m^{\langle k\rangle}=\binom{a_k+1}{k+1}+\binom{a_{k-1}+1}{k}+\dots +\binom{a_i+1}{i+1}\text{ and } m_{\langle k\rangle}=\binom{a_k-1}{k-1}+\binom{a_{k-1}-1}{k-2}+\dots +\binom{a_i-1}{i-1}.$$
	We also define $0^{\langle k\rangle}=0$ and $0_{\langle k\rangle}=0$ for $k\geq 1$.

	\begin{theorem}[$g$-theorem]
		An integer vector $h=(h_0, h_1, \dots, h_d)$ is the $h$-vector of a simplicial $(d-1)$-sphere if and only if 
		\begin{itemize}
			\item $h_i=h_{d-i}$ for all $0\leq i\leq d$;
			\item $1=h_0\leq h_1\leq \cdots \leq h_{\lfloor d/2\rfloor}$;
			\item The numbers $g_i:=h_i-h_{i-1}$ satisfy $g_{i+1}\leq g_i^{\langle i\rangle}$ for all $1\leq i\leq \lfloor d/2 \rfloor -1$.
		\end{itemize}
	\end{theorem}

	\section{Known upper bounds on the $m$-numbers and a few consequences}
	In light of the $g$-theorem, it is natural to try to characterize (or at least to find some necessary conditions on) the $m$-vectors of simplicial spheres. The following theorem (see \cite[Corollary 4.6(a)]{Nagel} and \cite[Corollaries 1.3 and 1.4]{MNZ}) provides tight upper bounds on the $m$-numbers in terms of the $g$-numbers. We refer to \cite{BilleraLee} for the definition of the Billera--Lee polytopes.
	
	\begin{theorem} \label{lem:upper-bounds}
		The $m$-numbers and the $g$-numbers of a simplicial $(d-1)$-sphere $\Delta$ satisfy the following inequalities.
		\begin{enumerate}
			\item For $1\leq k\leq \lceil d/2\rceil -1$, $m_k\leq g_k^{\langle k\rangle}-g_{k+1}$ while $m_{d-k}\leq g_k-(g_{k+1})_{\langle k+1\rangle}$. 
			\item If $d=2k$, then $m_k\leq g_k^{\langle k\rangle}+g_k$.
		\end{enumerate}
		These inequalities are tight: 
		\begin{enumerate}
			\item If $\Delta$ is the boundary complex of a Billera--Lee $d$-polytope, then the above inequalities are attained as equalities for all $k$.
			\item For $1\leq k\leq \lfloor d/2\rfloor -1$, $m_{d-k} = g_k$ if and only $\Delta$ is $k$-stacked. Moreover, if $d=2k+1$ and $\Delta$ is $k$-stacked, then $m_{k+1}=g_k$. 
		\end{enumerate} 
	\end{theorem}

In several special cases which we will now discuss, Theorem \ref{lem:upper-bounds} leads to a complete characterization of the $m$-vectors. Specifically, for $k$-neighborly $(d-1)$-spheres, the following holds.
	\begin{corollary} \label{cor: m-vectors of k-neighborly spheres}
		Let $\Delta$ be a $k$-neighborly $(d-1)$-sphere with $n\geq d+2$ vertices. Then 
		\begin{enumerate}
			\item $m_1=m_2=\dots =m_{k-1}=m_{d-k+1}=m_{d-k+2}=\dots=m_d=0$.
			\item If $d=2k$, then $m_k=\binom{n-k-1}{k+1}+\binom{n-k-2}{k}$.
			\item If $d=2k+1$, then $m_k=\binom{n-k-2}{k+1}$ and $0\leq m_{k+1}\leq \binom{n-k-3}{k}$.
		\end{enumerate}
		In particular, this gives a complete characterization of the $m$-vectors of neighborly $(2k-1)$-spheres.
	\end{corollary}
	
\smallskip\noindent {\it Proof (sketch):} A $k$-neighborly $(d-1)$-sphere $\Delta$ with vertex set $V$ has the complete $(k-1)$-skeleton, and hence it satisfies $m_i=0$ for all $i\leq k-1$. If $W$ is a missing $(d-i+1)$-face of $\Delta$ for some $i\leq k$, then $\beta_{d-i}(\Delta[W])=1$; here $\beta_{\ell}$ denotes the dimension of the $\ell$-th reduced simplicial homology. Then by the Alexander duality, $\beta_{i-2}(\Delta[V\backslash W])=1$. However, by $k$-neighborliness of $\Delta$, $\Delta[V\backslash W]$ has the complete $(k-1)$-skeleton, and hence $\Delta[V\backslash W]$ has vanishing $\beta_\ell$ for all $\ell\leq k-2$, including $\ell=i-2$. This contradiction shows that $m_{d-k+1}=m_{d-k+2}=\cdots=m_d=0$, and completes the proof of part 1. To finish the proof, observe that a  $k$-neighborly $(d-1)$-sphere with $n$ vertices has $g_k={n-d+k-2 \choose k}$ and  $m_k={n \choose k+1}-f_k$. Parts 2 and 3 then follow from Theorem \ref{lem:upper-bounds}  and direct computations.\endproof
	
	The $m$-vectors of $k$-stacked $(d-1)$-spheres have the opposite pattern: 
	since for such spheres, $g_i=0$ for all $k+1\leq i\leq d/2$,
	Theorem \ref{lem:upper-bounds} (alternatively, the Generalized Lower Bound Theorem \cite{MuraiNevo2013}) implies that the zero $m$-numbers are concentrated in the middle of the $m$-vector. 
	\begin{corollary}\label{cor: m-vectors of stacked spheres}
		Let $\Delta$ be a $k$-stacked $(d-1)$-sphere with $n\geq d+2$ vertices, where $0\leq k\leq \frac{d}{2}-1$. Then $m_{k+1}=m_{k+2}=\dots = m_{d-k-1}=0$; furthermore, if $0\leq k\leq \frac{d-1}{2}$, then $m_{d-k}=g_{k}$. In particular, 
		\begin{enumerate}
			\item if $d\geq 3$ and $\Delta$ is $1$-stacked, then $m_1=g_1^{\langle 1 \rangle}$, $m_{d-1}=g_1$, and all other $m$-numbers are zeros;
			\item if $k\geq 1$ and $\Delta$  is a $k$-stacked and $k$-neighborly $2k$-sphere, then $m_k=\binom{n-k-2}{k+1}$, $m_{k+1}= \binom{n-k-3}{k}$, and all other $m$-numbers are zeros.
		\end{enumerate}
	\end{corollary}
	\proof For any simplicial complex of dimension $d-1$ with $n$ vertices, $m_1=\binom{n}{2}-f_1$. If $\Delta$ is a stacked sphere, then $f_1=nd-\binom{d+1}{2}$, and so $m_1=g_1^{\langle 1 \rangle}$. Similarly, if $\Delta$ is a $k$-neighborly $2k$-sphere with $n$ vertices, then $g_k=\binom{n-k-3}{k}$. Therefore, if $\Delta$ is also $k$-stacked, then $m_{k+1}=g_k=\binom{n-k-3}{k}$. The other parts of the statement are immediate consequences of Theorem \ref{lem:upper-bounds}.
	\endproof
	
	Using Corollary \ref{cor: m-vectors of k-neighborly spheres}, we can provide a characterization of $m$-vectors of simplicial $2$-spheres, thus giving an answer to Problem \ref{Problem:Kalai} in the first non-trivial case of $d=3$. Recall that when $d=3$, any simplicial $2$-sphere is realizable as the boundary complex of a $3$-polytope.

	\begin{corollary} \label{cor:characterization of m for 2-spheres}
		An integer vector $m=(m_1, m_2, m_3)$ is the $m$-vector of a simplicial $2$-sphere with $n\geq 5$ vertices if and only if $m_1=g_1^{\langle 1\rangle}$, $0\leq m_2\leq g_1-2$ or $m_2=g_1$, and $m_3=0$.
	\end{corollary}
	\begin{proof}
		That $m_1=g_1^{\langle 1\rangle}$, $m_2\leq g_1$, and $m_3=0$ follows from the case $k=1$, $d=3$ of Corollary \ref{cor: m-vectors of k-neighborly spheres}. 
		
		A stacked $3$-polytope with $n$ vertices has $m_2=g_1$. To construct a simplicial $3$-polytope with $n$ vertices and $\ell$ missing $2$-faces for any $0\leq \ell\leq g_1-2$, consider the bipyramid over an $(n-\ell-2)$-gon. (Since $n-\ell-2\geq 4$, this bipyramid has no missing $2$-faces.) Now iteratively stack shallow pyramids on facets until a $3$-polytope with $n$ vertices is obtained. This requires $\ell$ stacking operations. The resulting polytope then has $\ell$ missing $2$-faces. 
		
		Finally, assume there exists a simplicial $3$-polytope $P$ with $n$ vertices and $g_1-1=n-5$ missing $2$-faces. Cutting $P$ along the $n-5$ planes affinely spanned by these missing $2$-faces, decomposes $P$ into $n-4$ simplicial $3$-polytopes with the total number of $n+3(n-5)=4(n-4)+1$ vertices. Hence $n-5$ of these polytopes must be simplices and the remaining polytope must have $5$ vertices, and so it must be the bipyramid over a triangle. Thus, the non-simplex polytope has a missing $2$-face, contradicting our assumption that $P$ had only $n-5$ missing $2$-faces.
	\end{proof}

	\section{Testing polytopality of neighborly spheres}
	The goal of this section is to show that $m$-numbers can be helpful for proving non-polytopality of some neighborly spheres.
	
	One of the most powerful methods that allows to construct a large number of neighborly polytopes and spheres is {\em sewing}.
	The idea is originally due to Shemer \cite{Shemer}. Let $d\geq 4$, let $\Delta$ be a neighborly $(d-1)$-sphere on the vertex set $[n]:=\{1, 2, \dots, n\}$, and let $B\subset \Delta$ be a $(\lfloor d/2\rfloor-1)$-neighborly and $(\lfloor d/2 \rfloor -1)$-stacked $(d-1)$-ball with $V(B)=[n]$. Then replacing $B$ in $\Delta$ with $\partial B *(n+1)$ results in a neighborly sphere with vertex set $[n+1]$. This operation is called an operation of sewing a new vertex onto $B$. Not all neighborly spheres are obtained by sewing (see Example \ref{ex:non-polytopality} below). However, as we will now show,  all {\em polytopal} neighborly spheres of odd dimension are obtained this way.
	
	When talking about subcomplexes, the following terminology will be useful. Assume $k\geq 1$. We say that a simplicial complex  $B$ is $k$-neighborly w.r.t.~a set $W$ if $B$ is $k$-neighborly and $V(B)=W$. We start with the following lemma. Parts of it were known before: the $d=4$ case is due to Perles (unpublished), while the case of any even $d$ is due to Bagchi and Datta \cite{BD2}. 
	
	\begin{lemma}\label{lm: stacked-links}
		Let $d\geq 4$ and let $P$ be a neighborly $d$-polytope with $f_0(P)\geq d+2$. Then for every vertex $v$ of $P$, the link of $v$ in $\partial P$ is  $(\lfloor d/2\rfloor -1)$-neighborly w.r.t.~$V(P)\backslash v$ and $(\lceil d/2\rceil -1)$-stacked.
	\end{lemma}
	\begin{proof}
		By slightly perturbing the vertices of $P$, we can assume that they have generic coordinates. Let $P'$ be the convex hull of all vertices of $P$ except $v$. Then $P'$ is a simplicial $d$-polytope and the complex generated by the facets of $P'$ that are {\em not} facets of $P$ provides a triangulation $T$ of $\lk(v, \partial P)$. Since $P$ is neighborly, every set of at most $ \lfloor d/2\rfloor$ vertices of $P'$ forms a face of $P$. Such face is either a face of the link of $v$ or an interior face of the antistar of $v$.\footnote{The antistar of $v$ is the subcomplex consisting of faces that do not contain $v$.} In either case, it is not an interior face of $T$. Consequently, $T$ has
		no interior faces of dimension $\leq \lfloor d/2 \rfloor-1$. In addition, $\lk(v, \partial P)$ must be $(\lfloor d/2\rfloor -1)$-neighborly because $P$ is $\lfloor d/2\rfloor$-neighborly. Hence $\lk(v, \partial P)$ is both $(\lfloor d/2\rfloor-1)$-neighborly and $(\lceil d/2 \rceil-1)$-stacked.
	\end{proof}
	
	\begin{corollary}
		Let $k\geq 2$ and let $P$ be a neighborly $2k$-polytope. Then $\partial P$ is obtained from the boundary complex of a $2k$-simplex by recursively sewing onto $(k-1)$-neighborly $(k-1)$-stacked balls.
	\end{corollary}
	\begin{proof}
		If $f_0(P)=2k+1$, then $P$ is a simplex, and the result holds. Otherwise, using the notation of the proof of Lemma \ref{lm: stacked-links}, $\partial P$ is obtained from $\partial P'$ by sewing vertex $v$ onto $T$. By the proof of Lemma \ref{lm: stacked-links}, $T$ is a $(k-1)$-neighborly $(k-1)$-stacked ball, and $P'$ is a neighborly $2k$-polytope with $f_0(P')=f_0(P)-1$. The statement then follows by induction on the number of vertices.
	\end{proof}
	
	Lemma \ref{lm: stacked-links}, together with part 2 of Corollary \ref{cor: m-vectors of stacked spheres}, provides a particularly simple numerical condition that any odd-dimensional neighborly polytopal sphere must satisfy. 
	This leads to:
	\begin{theorem} \label{non-polytopal}
		Let $k\geq 2$ and let $\Delta$ be a neighborly $(2k-1)$-sphere with $n\geq 2k+2$ vertices. If $\Delta$ has a vertex $v$ such that $m_k(\lk(v))\neq \binom{n-k-3}{k-1}$, then $\Delta$ is not polytopal.
	\end{theorem}
	
	\begin{example} 
		Any $3$-sphere with $7$ vertices is polytopal. The Gr\"unbaum--Sreedharan $3$-sphere $\mathrm{GS}_8$ from \cite{GruSree} is the only neighborly $3$-sphere with $8$ vertices that is not polytopal. That it is not polytopal is an immediate consequence of Theorem \ref{non-polytopal}. Indeed, the facets of $\mathrm{GS}_8$ are recorded in the following list:
		\begin{eqnarray*}\{1,2,3,4\},\{1,2,3,5\},
			\{1,2,4,5\},\{1,3,4,6\},\{1,3,5,6\},\{1,4,5,7\},\{1,4,6,7\}, \\ 
			\{1,5,6,8\},\{1,5,7,8\},\{1,6,7,8\},
			\{2,3,4,8\},\{2,3,5,6\},\{2,3,6,7\},\{2,3,7,8\},\\
			\{2,4,5,8\},\{2,5,6,8\},\{2,6,7,8\},\{3,4,6,7\},
			\{3,4,7,8\},\{4,5,7,8\}.
		\end{eqnarray*}
		In particular, two vertex links, namely, $\lk(4)$ and $\lk(6)$,  have $m_2=1$ instead of $\binom{8-2-3}{1}=3$, confirming that $\mathrm{GS}_8$ is not polytopal. (In fact, these two links are isomorphic to the connected sum of an octahedral sphere and the boundary complex of a $3$-simplex.) 
	\end{example}

	\begin{remark}
		Theorem \ref{non-polytopal} can be extended to the case that $\Delta$ is a simplicial $(2k-1)$-sphere with $n\geq 2k+2$ vertices and $f_{k-1}=\binom{n}{k}-1$, i.e., $\Delta$ is $(k-1)$-neighborly and it has exactly one missing $(k-1)$-face $F$. 
		Using the same proof as in Lemma \ref{lm: stacked-links}, we can show that if $\Delta$ is polytopal, then for any vertex $v\in F$, the link of $v$ must be $(k-1)$-stacked,  and consequently, $m_{k}(\lk(v))=g_{k-1}(\lk(v))=\binom{n-k-3}{k-1}-1$. For an application of this observation, consider the Barnette sphere \cite{Barnette-sphere} --- the only non-neighborly simplicial $3$-sphere with $8$ vertices that is not polytopal. This sphere has a single missing edge $e$, and the link of any of the endpoints of $e$ satisfies $m_2=0$ instead of $\binom{8-2-3}{1}-1=2$, confirming that the sphere is not polytopal. (The two links are octahedral spheres.) 
	\end{remark}
	
	\begin{example} \label{ex:non-polytopality}
		Using Theorem \ref{non-polytopal}, one can check that the following vertex-transitive  neighborly $3$- and $5$-spheres from Frank Lutz's Manifold page \cite{manifold_page} are not polytopal.  (The first two numbers indicate the dimension and the number of vertices, respectively; for example, $3\_10\_1\_1$ is a $10$-vertex triangulation of the $3$-sphere. Since the complexes are vertex-transitive, the values of $m_2(\lk(v)$ and $m_3(\lk(v)$ are independent of the choice of vertex $v$.)
		
		\begin{tabular}{|c|c|c|c|c|c|}
			\hline
			Manifold & $3\_10\_1\_1$ & $3\_11\_1\_1$ & $3\_13\_1\_3$ & $3\_13\_1\_5$ & $3\_14\_1\_7$ \\
			\hline	
			$m_2(\lk(v))$ & 3 & 3 & 2 & 3 & 5\\
			\hline
			Manifold & $3\_14\_1\_8$ & $3\_14\_1\_11$ & $3\_14\_1\_14$ & $3\_14\_1\_17$ & $3\_14\_1\_18$ \\
			\hline	
			$m_2(\lk(v))$  & 7 & 4 & 7 & 6 & 7 \\
			\hline
			Manifold &  $3\_14\_1\_26$ &  $3\_14\_1\_27$ & $3\_15\_1\_3$ & $3\_15\_1\_13$ &   \\
			\hline	
			$m_2(\lk(v))$ & 7 & 5 & 6 & 5 &  \\
			\hline
			Manifold & $5\_11\_1\_1$ & $5\_13\_2\_6$ & $5\_13\_1\_8$ & $5\_15\_2\_7$ &  \\
			\hline	
			$m_3(\lk(v))$ & 8 & 15 & 11 & 24 & \\
			\hline
		\end{tabular}
	\end{example}
	
	One key ingredient of the proof of Theorem \ref{non-polytopal} is that any $(k-1)$-neighborly $(k-1)$-stacked $(2k-2)$-sphere with $n-1$ vertices has $m_k=\binom{n-k-3}{k-1}$; see part 2 of Corollary \ref{cor: m-vectors of stacked spheres}. While no similar results are known for $(k-1)$-neighborly $k$-stacked $(2k-1)$-spheres, 
	any $\ell$-stacked $(d-1)$-sphere must have a missing face of dimension  $\geq d-\ell$. (This follows from the definition of $\ell$-stackedness.) An immediate consequence of this observation and Lemma \ref{lm: stacked-links} is the following relative of Theorem~\ref{non-polytopal}.

	\begin{corollary}
		Let $k\geq 2$, $d\in\{2k,2k+1\}$, and let $\Delta$ be a neighborly $(d-1)$-sphere. If there is a vertex $v$ of $\Delta$ such that all missing faces of $\lk(v)$ have dimension $k-1$, then $\Delta$ is not polytopal.
	\end{corollary}
	
	\noindent For instance, a flag $4$-polytope ($3$-polytope, resp.) is not a vertex figure of any neighborly $5$-polytope ($4$-polytope, resp.) This discussion motivates the following questions.
	\begin{question}  \label{question:m_{k+1}=0}
		For $k\geq 2$ and $d\in \{2k+1, 2k+2\}$, are there $k$-neighborly $d$-polytopes, with arbitrarily many vertices, all of whose missing faces have dimension $k$? 
	\end{question}
	In Section 7 we will prove that for every {\em odd} $k$ and $n\geq 2k+4$ as well as for $k=2$ and $n\geq 9$, there exists a neighborly $(2k+1)$-polytope with $n$ vertices all of whose missing faces have dimension $k$. 
	The question remains open in all other cases. The complexes $7\_12\_193\_1$ and $7\_13\_1\_1$ from \cite{manifold_page} are $3$-neighborly $7$-spheres all of whose missing faces have dimension $3$. We do not know whether they are polytopal or not.

	\begin{question} \label{quest:missing-faces-dim-k}
		For $k\geq 1$, are there neighborly $(2k+1)$- and $(2k+2)$-spheres all of whose vertex links have missing faces only in dimension $k$?  If so, can we find such $(2k+1)$-and $(2k+2)$-spheres with  arbitrarily many vertices? 
	\end{question}
	
	The following examples of neighborly manifolds from \cite{manifold_page} suggest that the answers to Question \ref{quest:missing-faces-dim-k} might be positive. The complex $3\_15\_11\_1$ is a $15$-vertex neighborly triangulation of the $3$-torus all of whose vertex links are flag $2$-spheres.
	Similarly, the complex $5\_13\_3\_2$ is a $13$-vertex $3$-neighborly triangulation of  $\mathrm{SU}(3)/\mathrm{SO}(3)$. All vertex links of this complex are $2$-neighborly $4$-spheres  and all missing faces of these links have dimension $2$. 
	Both triangulations are vertex-transitive.

	\section{Lower bounds on the $m$-numbers}
	This section is devoted to establishing lower bounds on the $m$-numbers. First, we provide an extension of Goodman's bound to all simplicial complexes. Then we discuss applications of this bound to nearly neighborly spheres.
	
	\subsection{Extending Goodman's bound}
	For $s\geq 3$, denote by $G_s(n, f_1)$ the minimum number of $s$-cliques that a graph with $n$ vertices and $f_1$ edges can have. Expressing the value  $G_s(n, f_1)$ in terms of $n$ and $f_1$ is known  in the extremal combinatorics as the clique density problem. For spectacular recent developments on the {\em tight} lower bound on $G_s(n, f_1)$, 
	see \cite{LPS, Nikoforov,Razborov,Reiher}. In the case of $s=3$, the following theorem provides a simple convex lower bound on $G_3(n, f_1)$; it is known as Goodman's bound \cite{Goodman, MoonMoser,NordhausStewart}. Denote by $T_r(n)$ the Tur\'an graph, i.e., the complete $r$-partite graph with $n$ vertices partitioned into $r$ parts of sizes as equal as possible; also denote by $t_r(n)$ the number of edges of $T_r(n)$.
	\begin{theorem}  \label{thm:Goodman-bound}
		$G_3(n, f_1)\geq \frac{f_1(4f_1-n^2)}{3n}$. The equality holds if and only if $f_1=t_r(n)$ for some $r$ that divides $n$, with the graph $T_r(n)$ attaining the bound.
	\end{theorem}
	If $\Delta$ is a simplicial complex, then a $3$-clique of the graph of $\Delta$ is either a $2$-face or a missing $2$-face. Thus, the number of $3$-cliques of the graph of $\Delta$ is $f_2+m_2$. Theorem \ref{thm:Goodman-bound} then implies
	
	\begin{corollary} \label{cor:Goodman-bound-on-m_2}
		Let $\Delta$ be a simplicial complex of dimension $\geq 1$ and with $n$ vertices. Then $m_2\geq \frac{f_1(4f_1-n^2)}{3n}-f_2.$	
	\end{corollary}

	Our first result is an extension of Corollary \ref{cor:Goodman-bound-on-m_2} to a bound on $m_k$ for all simplicial complexes.
	
	\begin{theorem}\label{lm: k-1-neighborly higher dimensional spheres}
		Let $k\geq 2$, and let $\Delta$ be a simplicial complex of dimension $\geq k-1$ and  with $n$ vertices. Then
		$$ 
		m_k\geq \frac{k^2}{(k+1)\binom{n}{k-1}}f_{k-1}^2-\frac{n(k-1)-k(k-2)}{k+1}f_{k-1}-f_k.
		$$
	\end{theorem}
	\noindent Note that when $k=2$, Theorem \ref{lm: k-1-neighborly higher dimensional spheres} reduces to Goodman's bound (see Corollary \ref{cor:Goodman-bound-on-m_2}).
	In fact, our proof is very similar to the proof of Theorem \ref{thm:Goodman-bound} given in \cite{Goodman}. Also, similarly to Goodman's bound, the lower bound on $m_k+f_k$ from Theorem \ref{lm: k-1-neighborly higher dimensional spheres} is non-trivial (i.e., positive) only when $f_{k-1}$ is large.
	
	\begin{proof}
		Let $S$ be the collection of $k$-subsets of $[n]$ that are {\bf not} faces of $\Delta$. Then $|S|=\binom{n}{k}-f_{k-1}$.
		For a $(k-1)$-subset $L$ of $[n]$, let $S_L=\{s\in S: L\subset s\}$. By double counting, $k|S|=\sum_{L\subset [n], |L|=k-1}|S_L|$.
		
		Let $T$ consist of those $(k+1)$-subsets of $[n]$ that contain at least one element of $S$ as a subset. In particular, $f_k+m_k+|T|=\binom{n}{k+1}$.
		For $1\leq i\leq k+1$, define $$T_i=\{t\in T:\text{exactly $i$ elements of $S$ are subsets of $t$} \}. $$
		Then
		\begin{equation}\label{eq1}
			|T|=(n-k)|S|-\sum_{i=2}^{k+1}|T_i|(i-1).
		\end{equation} 
		
		Consider $L\subset[n]$, $|L|=k-1$. Observe that the union of any two elements of $S_L$ is in one of $T_i$'s, and that all such unions are distinct.  On the other hand, any element $C$ of $T_i$ contains exactly $i$ elements of $S$. The union of every two of these $i$ elements is $C$ and every two such elements belong to a unique $S_L$. Therefore, by double counting,
		\begin{equation}  \label{Cauchy-Schwarz}
			\begin{split}
				\sum_{i=2}^{k+1}|T_i|\binom{i}{2}&=\sum_{L\subset [n], |L|=k-1}\binom{|S_L|}{2}=\frac{1}{2}\left(-k|S|+\sum_{L\subset [n], |L|=k-1}|S_L|^2\right) \\
				&\geq -\frac{k|S|}{2}+\frac{1}{2\binom{n}{k-1}}\left(\sum_{L\subset [n], |L|=k-1}|S_L|\right)^2
				=\frac{k^2{|S|^2}}{2\binom{n}{k-1}}-\frac{k|S|}{2},
			\end{split}
		\end{equation}
		where the penultimate step is by the Cauchy--Schwarz inequality. Thus,
		\begin{equation}\label{eq2}
			\sum_{i=2}^{k+1}|T_i|(i-1)\geq \frac{2}{k+1}\sum_{i=2}^{k+1}|T_i|\binom{i}{2}\geq \frac{k^2|S|^2}{(k+1)\binom{n}{k-1}}-\frac{k|S|}{k+1}.
		\end{equation}
		Combining \eqref{eq1} and \eqref{eq2}, we obtain
		\begin{equation} \label{eq:|T|}
			|T|\leq (n-k)|S|-\left(\frac{k^2|S|^2}{(k+1)\binom{n}{k-1}}-\frac{k|S|}{k+1}\right). 
		\end{equation} 
		Recall that $|T|=\binom{n}{k+1}-f_k-m_k$ and $|S|=\binom{n}{k}-f_{k-1}$. Substituting these expressions in eq.~\eqref{eq:|T|} and simplifying the coefficients,  implies the promised lower bound on $m_k$.
	\end{proof}
	
	Analyzing equations \eqref{Cauchy-Schwarz} and \eqref{eq2}, we observe that the inequality of Theorem \ref{lm: k-1-neighborly higher dimensional spheres} holds as equality if and only if 1) all the sets $S_L$, as $L$ ranges over $(k-1)$-subsets of $[n]$, have the same size,  and 2) all $T_i$, with $2\leq i\leq k$, are empty sets. When $k=2$ this means that equality holds if and only all vertices have the same degree and $T_2=\emptyset$, which easily implies that the graph of $\Delta$ is $T_r(n)$ for some $r$ that divides $n$; see Theorem \ref{thm:Goodman-bound}. 
	When $k=3$, as an example that attains the bound, take any $3$-dimensional $2$-neighborly complex on vertex set $[7]$ whose set of missing $2$-faces consists of $$\{1,2,3\}, \{1,4,5\}, \{1,6,7\}, \{2,4,6\}, \{2,5,7\},\{3,4,7\}, \{3,5,6\}.$$ In this example, the missing $2$-faces correspond to flats of size $3$ of the Fano matroid.
	
	\subsection{Nearly neighborly Eulerian complexes}
	We now discuss an application of Theorem \ref{lm: k-1-neighborly higher dimensional spheres} to nearly neighborly Eulerian complexes.
	The {\em reduced Euler characteristic} of a simplicial complex $\Delta$ is $\tilde{\chi}(\Delta):=\sum_{i=-1}^{\dim \Delta} (-1)^i f_i(\Delta)$. A simplicial complex $\Delta$ is called {\em Eulerian} if for every face $F$ of $\Delta$, including the empty face, $\tilde{\chi}(\lk(F, \Delta))=(-1)^{\dim \Delta -\dim F-1}$. For instance, all simplicial spheres are Eulerian and so are all odd-dimensional simplicial manifolds. Eulerian complexes were introduced by Klee in \cite{Klee64}. Klee proved that if $\Delta$ is Eulerian of dimension $d-1$, then $\Delta$ satisfies the Dehn--Sommerville equations, namely, $h_i(\Delta)=h_{d-i}(\Delta)$ for all $i$. Using these relations leads to the following restatement of Theorem \ref{lm: k-1-neighborly higher dimensional spheres} for nearly neighborly Eulerian complexes. As with spheres, we  say that a $(d-1)$-dimensional Eulerian complex $\Delta$ is {\em nearly neighborly} if it is $(\lfloor d/2\rfloor-1)$-neighborly, and  that $\Delta$ is {\em neighborly} if it is $\lfloor d/2\rfloor$-neighborly. In particular, all $3$- and $4$-dimensional Eulerian complexes are nearly neighborly.

	\begin{corollary} \label{lm: near neighborly spheres}
		Let $k\geq 2$, $d\in\{2k,2k+1\}$, and
		let $\Delta$ be a $(d-1)$-dimensional nearly neighborly Eulerian complex with $n$ vertices. Then 
		\begin{equation*} 
			\begin{split}
				m_k(\Delta) \geq\;&  \frac{k^2}{(k+1)\binom{n}{k-1}}f_{k-1}(\Delta)^2-\frac{n(k-1)-k(k-2)}{k+1}f_{k-1}(\Delta) \\
				&-f_k(C(d,n))+\left(\lfloor d/2\rfloor+1+(-1)^{d-1}\right)\left(\binom{n}{k}-f_{k-1}(\Delta)\right).
			\end{split}
		\end{equation*}
	\end{corollary}
	\begin{proof}
		By Theorem \ref{lm: k-1-neighborly higher dimensional spheres}, to prove the statement, it suffices to show that $$f_k(\Delta)=f_k(C(d,n))-\left(\lfloor d/2\rfloor+1+(-1)^{d-1}\right)\left(\binom{n}{k}-f_{k-1}(\Delta)\right).$$ As we will see, this is an easy consequence of 
		$(k-1)$-neighborliness and Dehn--Sommerville relations. 
		
		Assume first that $d=2k$. Since $C(2k,n)$ is $k$-neighborly and $\Delta$ is $(k-1)$-neighborly, $$f_{i-1}(\Delta)=f_{i-1}(C(2k, n))  \mbox{ for all } i\leq k-1 \quad\mbox{and} \quad f_{k-1}(\Delta)=f_{k-1}(C(2k, n))-m_{k-1}(\Delta).$$ Hence $$h_i(\Delta)=h_i(C(2k, n)) \mbox{ for all } i\leq k-1 \quad\mbox{and} \quad h_k(\Delta)=h_k(C(2k, n))-m_{k-1}(\Delta). $$ By the Dehn--Sommerville relations, $h_{k+1}(\Delta)=h_{k-1}(\Delta)$ and $h_{k+1}(C(2k,n))=h_{k-1}(C(2k,n))$, and so
		\begin{eqnarray*}
			f_k(\Delta) &=&\sum_{i=0}^{k+1}\binom{2k-i}{k-1}h_i(\Delta)=\left[\sum_{i=0}^{k+1}\binom{2k-i}{k-1}h_i(C(2k, n))\right]-km_{k-1}(\Delta)\\
			&=&f_k(C(2k,n))-km_{k-1}(\Delta)=f_k(C(2k,n))-k\left(\binom{n}{k}-f_{k-1}(\Delta)\right),
		\end{eqnarray*}
		as desired.
		
		The case of $d=2k+1$ is similar. In this case, the Dehn--Sommerville relations imply that $h_{k+1}(\Delta)=h_{k}(\Delta)=h_k(C(2k+1, n))-m_{k-1}(\Delta)=h_{k+1}(C(2k+1, n))-m_{k-1}(\Delta)$, and hence
		\begin{eqnarray*}
			f_k(\Delta)&=&\sum_{i=0}^{k-1}\binom{2k+1-i}{k}h_i(\Delta)+(k+1)h_k(\Delta)+h_{k+1}(\Delta)\\
			&=&f_k(C(2k+1,n))-(k+2)m_{k-1}(\Delta)=f_k(C(2k+1,n))-(k+2)\left(\binom{n}{k}-f_{k-1}(\Delta)\right).
		\end{eqnarray*}
		The result follows.
	\end{proof}
	
	The class of $3$- and $4$-dimensional Eulerian complexes deserves special attention. In this case, the Dehn--Sommerville relations imply that $f_2=2(f_1-f_0)$ if dimension is $3$, and $f_2=4f_1-10f_0+20$ if dimension is $4$. Thus, Corollary \ref{cor:Goodman-bound-on-m_2} (or Corollary \ref{lm: near neighborly spheres}) can be rewritten as follows:
	
	\begin{corollary} \label{cor:3-and-4-Eulerian}
		Let $\Delta$ be a $(d-1)$-dimensional Eulerian complex with $n$ vertices. Then 
		$m_2\geq \frac{f_1(4f_1-n^2)}{3n}-2(f_1-n)$ if $d=4$, and $m_2\geq \frac{f_1(4f_1-n^2)}{3n}-(4f_1-10n+20)$ if $d=5$. 
	\end{corollary}
	
	If $\Delta$ is also flag (or, more generally, a complex with $m_2=0$), then Corollary \ref{cor:3-and-4-Eulerian} leads to the following upper bound on the number of edges of $\Delta$.

	\begin{corollary} \label{cor:flag}
		Let $\Delta$ be a $(d-1)$-dimensional flag Eulerian complex with $n$ vertices.
		Then $f_1<n^2/4+3n/2$ if $d=4$, and $f_1<n^2/4+3n$ if $d=5$.
	\end{corollary}
	\begin{proof}
		If $d=4$, then by Corollary \ref{cor:3-and-4-Eulerian},  $m_2\geq  \frac{f_1(4f_1-n^2)}{3n}-2(f_1-n)$. Solving the inequality $\frac{f_1(4f_1-n^2)}{3n}-2(f_1-n)\geq 1$ w.r.t.~$f_1$, we conclude that if $f_1\geq n^2/4 + 3n/2$, then $m_2\geq 1$ and hence the complex is not flag. 
		Similarly, if $d=5$, solving the inequality $\frac{f_1(4f_1-n^2)}{3n}-(4f_1-10n+20) \geq 1$ implies that no $4$-dimensional Eulerian complex with $n$ vertices and $f_1\geq n^2/4+3n$ can be flag.
	\end{proof}
	
	In fact, it is proved in \cite{Z-flag} that any $3$-dimensional flag Eulerian complex with $n$ vertices must satisfy $f_1\leq \lfloor n^2/4\rfloor +n$ and this upper bound is tight. In other words, in dimension $3$, the upper bound on $f_1$ produced by our methods is not tight, but it is not too far from being tight. On the other hand, the upper bound on $f_1$ (and hence also on all other face numbers) for flag Eulerian complexes of dimension $4$ appears to be new. It is conjectured that a flag $4$-sphere with $n$ vertices has at most $\lfloor n^2/4\rfloor+2n-5$ edges; see \cite[Conjecture 18]{Adam-Hladky}.
	
	Corollary \ref{cor:3-and-4-Eulerian} also leads to the following asymptotic bound.
	Let $\Delta$ be a $3$-dimensional Eulerian complex with $f_0=n$ vertices and $f_1$ edges. Assume, $f_1=\lambda n^2 +\mu n^\alpha + o(n^\alpha)$, where $0\leq \lambda \leq 1/2$ and $0\leq \alpha <2$. Then
	\begin{equation*}
		\begin{split}
			m_2 &\geq \frac{f_1(4f_1-n^2)}{3n}-2(f_1-n)\\
			&=\frac{1}{3}\left( \lambda n +\mu n^{\alpha-1}\right) \left( 4\lambda n^2 +4\mu n^\alpha -n^2\right) - (2\lambda n^2-2n) +o(n^{\alpha+1}) \\
			&=\begin{cases}
				\frac{\lambda(4\lambda-1)}{3}n^3 +o(n^3) & \text{if }1/2\geq \lambda>1/4 \\
				\frac{\mu n^{\alpha+1}}{3}+o(n^{\alpha+1}) &\text{if } \lambda=1/4, \, 2>\alpha>1\\
				\frac{2\mu-3}{6}n^2+o(n^2) & \text{if }\lambda=1/4, \, \alpha=1, \mu \geq 3/2\\
				0 & \text{otherwise}		
			\end{cases}.
		\end{split}
	\end{equation*}
	In the same vein, Corollary \ref{lm: near neighborly spheres} implies that if $k\geq 2$ and $d\in\{2k,2k+1\}$, then for $\epsilon>0$, any nearly neighborly Eulerian complex of dimension $d-1$  with $n\gg 0$ vertices and $f_{k-1} \geq (\frac{k-1}{k}+\epsilon)\binom{n}{k}$ has $m_k \geq  C_\epsilon n^{k+1} +O(n^k)$, where $C_\epsilon$ is some positive constant that depends only on $k$ and $\epsilon$.
	
	We close this section with a few open problems.
	For fixed $f_0$ and $f_1$, denote by $\mathcal{C}(f_0, f_1)$ the class of graphs with $f_0$ vertices and $f_1$ edges that attain the minimum number $G_3(f_0, f_1)$  of $3$-cliques. 
	When $f_0$ is large and the edge density $f_1/\binom{f_0}{2}$ is bounded away from $1$, a complete characterization of $\mathcal{C}(f_0, f_1)$ can be found in \cite{LPS}. For instance, when $f_1=t_r(f_0)$ for some $r$ that divides $f_0$, $\mathcal{C}(f_0, f_1)=\{T_r(f_0)\}$. The discussion of this section shows that for $d\in \{4,5\}$, characterizing the simplicial $(d-1)$-spheres that attain the lower bound on $m_2$ in terms of $f_0$ and $f_1$ is equivalent to characterizing the values of $f_0, f_1$, and the graphs in $\mathcal{C}(f_0, f_1)$ that can be realized as the graphs of simplicial $(d-1)$-spheres.
	This leads to the following problem.
	
	\begin{problem} \label{problem:turan}
		Let $d\in \{4,5\}$. Characterize simplicial $(d-1)$-spheres with $f_0$ vertices and $f_1$ edges  whose graphs are in $\mathcal{C}(f_0, f_1)$. In particular, for which values of $r$ and $f_0$ can the Tur\'an graph $T_r(f_0)$ be realized as the graph of a simplicial $(d-1)$-sphere?
	\end{problem}
	\noindent Not much is known. For part 1, the existence of neighborly $3$- and $4$-spheres shows that the complete graph $T_1(n)$ is the graph of a $3$-sphere for all $n\geq 5$ and it is the graph of a $4$-sphere for all $n\geq 6$. Similarly, the existence of centrally symmetric (cs, for short) $3$- and $4$-spheres with $2n$ vertices that are cs-$2$-neighborly (i.e., every two non-antipodal vertices are connected by an edge) implies that $T_n(2n)$ is the graph of a $3$-sphere for all $n\geq 4$ and it is also the graph of a $4$-sphere for all $n\geq 5$; see \cite{Jockusch95, NZ-cs-neighborly-new}. Another result along these lines is from \cite{Z-balanced}:  it is shown there that $T_4(16)$ is the graph of a $3$-sphere while $T_4(12)$ is not.
	\begin{problem}
		Let $d\in \{4,5\}.$ Let $$U_d(f_0, f_1)=\begin{cases}
			g_2^{\langle 2\rangle}+g_2 & \text{if } d=4 \\
			g_2^{\langle 2\rangle} & \text{if } d=5
		\end{cases},$$
		$$\text{and let}\quad L_d(f_0, f_1)=\begin{cases}
			G_3(f_0, f_1)-2(f_1-f_0) & \text{if } d=4 \\
			G_3(f_0, f_1)-(4f_1-10f_0+20) & \text{if } d=5
		\end{cases}.$$
		By Theorem \ref{lem:upper-bounds} and our discussion in this section, $U_d(f_0, f_1)$ is the maximum value of $m_2$ that a simplicial $(d-1)$-sphere with $f_0$ vertices and $f_1$ edges can have, while $L_d(f_0, f_1)$ is a lower bound on possible values of $m_2$. Which integers $m$ between $L_d(f_0, f_1)$ and $U_d(f_0, f_1)$, can be realized as an $m_2$-number of simplicial $(d-1)$-spheres with $f_0$ vertices and $f_1$ edges?
	\end{problem}
	
	\noindent As an example, in the case of $d=4$, $f_0=10$, and $f_1=\binom{10}{2}-5=40$, the upper bound on $m_2$ is attained by the Billera--Lee $4$-polytope with $10$ vertices and $5$ missing edges, while the lower bound is attained by a centrally symmetric $3$-sphere with $10$ vertices whose graph is $T_5(10)$. 
Hence $U_4(10, 40)=30$ and $L_4(10, 40)=20$ are both realizable as $m_2$ of a $3$-sphere with $10$ vertices and $40$ edges.  We do not know which integers between $20$ and $30$ are also realizable. 

\section{The $m$-vectors of neighborly $4$-spheres}
To start our discussion of the $m$-vectors of neighborly $4$-spheres, recall that by Corollary \ref{cor: m-vectors of k-neighborly spheres}, 
all $m$-numbers of neighborly $2k$-spheres with $n$ vertices, except $m_{k+1}$, are fixed functions of $n$ and $k$, while $m_{k+1}$ could vary and is upper bounded by $\binom{n-k-3}{k}$. This motivates the following
\begin{question}\label{question: m_{k+1} of neighborly 2k-spheres}
	Let $k\geq 2$. For any sufficiently large $n$ and any $m$ between $0$ and $\binom{n-k-3}{k}$, are there neighborly $2k$-spheres with $n$ vertices and with $m_{k+1}=m$?
\end{question}
By far the largest family of neighborly $2k$-spheres with $n$ vertices was constructed in \cite{NZ-neighborly}; the construction is given by {\em relative squeezed spheres}. Since all these spheres are $k$-stacked, by Corollary \ref{cor: m-vectors of stacked spheres}, they all have $m_{k+1}=\binom{n-k-3}{k}$. On the other extreme is the question of whether there exist neighborly $2k$-spheres with $m_{k+1}=0$ (see Question \ref{question:m_{k+1}=0}). The following theorem, whose proof we defer until Section 7, partially answers this question.

\begin{theorem}\label{thm: neighborly 2k-sphere m_k+1=0}
	For all odd $k\geq 3$ and any $n\geq 2k+4$ as well as for $k=2$ and $n\geq 9$, there exists a neighborly $(2k+1)$-polytope with $n$ vertices all of whose missing faces have dimension $k$.
\end{theorem}
The goal of this section is to settle Question \ref{question: m_{k+1} of neighborly 2k-spheres} in the case of $k=2$ except for a single value of $m$. Specifically, we prove the following result (cf.~Corollary \ref{cor:characterization of m for 2-spheres}).

\begin{theorem}\label{thm: neighborly 4-spheres}
	For any $n\geq 9$ and $0\leq m\leq \binom{n-5}{2}$ with $m\neq \binom{n-5}{2}-1$, there exists a neighborly PL $4$-sphere with $n$ vertices and $m_3=m$.
\end{theorem}

In the case of $m=0$, the result follows from Theorem \ref{thm: neighborly 2k-sphere m_k+1=0}. To  prove the result for positive values of $m$, we start with the boundary complex of the cyclic $5$-polytope. In the first part of our construction, we apply to $\partial C(5,n)$ a sequence of bistellar flips that reduce $m_3$ but preserve neighborliness. A bistellar flip is defined as follows. 
If $\Delta$ is a PL $(d-1)$-sphere that contains an induced subcomplex $\overline{A}*\partial \overline{B}$, where $A$ is a $j$-subset of $V(\Delta)$ and $B$ is a $(d-j+1)$-subset of $V(\Delta)$ for some $1\leq j\leq d$, then one can perform a {\em bistellar flip} on $\Delta$ by replacing $\overline{A}*\partial \overline{B}$ with $\partial \overline{A}*\overline{B}$. (In this case we say that we apply the bistellar flip on the star of $A$.) The resulting complex is another PL $(d-1)$-sphere.\footnote{To give some examples of PL spheres, we notice that the boundary complex of any simplicial polytope is a PL sphere. So is any shellable sphere as well as the boundary complex of any shellable ball.} 

Assume that $n\geq 7$. Denote by $B(4, [2, n-1])$ the complex generated by the facets of $\partial C(4, n)$ of the form $\{i_1,i_1+1,i_2,i_2+1\}$, where $2\leq i_1<i_1+1<i_2< i_2+1\leq n-1$. It follows from the Gale evenness condition that $\partial B(4,[2,n-1])$ is the boundary complex of the cyclic $3$-polytope with vertex set $[2,n-1]:=\{2,3,\dots,n-1\}$ and that $\partial(\overline{1n}*B(4, [2, n-1]))=\partial C(5, n)$. In particular, the complex $\overline{1n}*B(4, [2,n-1])$ is a $2$-stacked $5$-ball. As such, it has no missing $3$-faces (see \cite[Theorem 2.3]{MuraiNevo2013}). Thus, all missing $3$-faces of $\partial C(5,n)$ are the minimal interior $3$-faces of $\overline{1n}*B(4, [2,n-1])$, that is, they are all of the form $\{1, n\}\cup H$, where $H$ is a missing edge of $\partial B(4, [2, n-1])$. In short, the set of missing $3$-faces of $\partial C(5,n)$ is given by $$M_1=\{\{1,i, j, n\}: 3\leq i, j\leq n-2, j-i\geq 2\}.$$

We now define a certain collection of PL $4$-spheres $\Delta_i$ for $1\leq i\leq n-6$.
\begin{definition}  \label{def:Delta_i}
	Let $n\geq 8$ and let $\Delta_1:=\partial C(5,n)$. For $2\leq i\leq n-6$, assume that $\Delta_{i-1}$ is already defined, that $\Delta_{i-1}$ is a  PL $4$-sphere with $n$ vertices,  and that $\st(\{1, i+1, n\}, \Delta_{i-1})=\partial \overline{\{2,i+2,n-1\}}*\overline{\{1,i+1,n\}}$ is an induced subcomplex of $\Delta_{i-1}$. Define $\Delta_i$ as the complex obtained from $\Delta_{i-1}$ by applying the bistellar flip on $\st(\{1,i+1,n\}, \Delta_{i-1})$. Also, for $1\leq j\leq n-6$, define $M_j$ as the set of missing $3$-faces of $\Delta_j$.
\end{definition}

To justify this definition, we inductively prove the following result.

\begin{proposition} \label{prop:Delta_i}
	For $2\leq i\leq n-6$, the complex $\Delta_i$ is well-defined, and it is a neighborly PL $4$-sphere with $n$ vertices. Furthermore, if we let $S_{i-1}=\{\{1, i+1,j, n\}: i+3\leq j\leq n-2\}$, then $M_i=M_{i-1}\backslash S_{i-1}$. In particular, $m_3(\Delta_i)=\binom{n-4-i}{2}$ for all $1\leq i\leq n-6$.
\end{proposition}

\begin{proof}
	To start, note that the first part of the statement holds for $i=2$. Indeed, $\st(\{1,3,n\}, \Delta_1)=\partial \overline{\{2,4,n-1\}}*\overline{\{1,3,n\}}$ is an induced subcomplex of $\Delta_1$, and so $\Delta_2$ is well-defined. Since $\Delta_2$ is obtained from a neighborly PL $4$-sphere by a flip that does not affect the set of edges, it is also a neighborly PL $4$-sphere. 
	
	The set of facets of $\lk(1n, \Delta_1)$ consists of $\{2,3,n-1\}, \{2,n-2, n-1\},$ and $\{2, j, j+1\}, \{j, j+1, n-1\}$ for $3\leq j\leq n-3$. Assume inductively that the set of facets of $\lk(1n, \Delta_i)$ consists of $\{2,i+2,n-1\}, \{2,n-2, n-1\},$ and $\{2, j, j+1\}, \{ j, j+1, n-1\}$ for $i+2\leq j\leq n-3$, and that $M_i=M_{1}\backslash (\cup_{1\leq k\leq i-1} S_k).$  The first assumption guarantees that $\st(\{1,i+2,n\}, \Delta_{i})=\partial \overline{\{2,i+3, n-1\}}*\overline{\{1,i+2,n\}}$. Furthermore, $\{2, i+3, n-1\}$ is a missing $2$-face of $\Delta_1$, and it was not added as a $2$-face in any of the bistellar flips we performed to get from $\Delta_1$ to $\Delta_i$. Hence  $\st(\{1,i+2,n\}, \Delta_{i})$ is an induced subcomplex of $\Delta_i$. Thus $\Delta_{i+1}$ is well-defined. In particular, $\lk(1n, \Delta_{i+1})$ is obtained from $\lk(1n,\Delta_i)$ by replacing $(i+2)*\partial \overline{\{2, i+3, n-1\}}$ with $\overline{\{2, i+3, n-1\}}$. Hence the facets of $\lk(1n, \Delta_{i+1})$ are $\{2,i+3,n-1\}, \{2,n-2, n-1\},$ and $\{2, j, j+1\}, \{j, j+1, n-1\}$ for $i+3\leq j\leq n-3$. All missing $3$-faces of $\Delta_i$ containing $\{1, i+2, n\}$ --- this is precisely the set $S_i$ --- are no longer missing $3$-faces of $\Delta_{i+1}$. Furthermore, the newly added $2$-face $\{2, i+3, n-1\}$ has not created any missing $3$-faces (as the only vertices of $\Delta_{i+1}$ whose link contains $\partial \overline{\{2,i+3, n-1\}}$ are $1$, $i+2$, and $n$). This proves that $M_{i+1}=M_{1}\backslash (\cup_{1\leq k\leq i} S_k).$
	
	Finally, $m_3(\Delta_i)=|M_1|-\sum_{1\leq k\leq i-1}|S_k|=\binom{n-5}{2}-\sum_{1\leq k\leq i-1}(n-5-k)=\binom{n-4-i}{2}$. 
\end{proof}
\begin{remark}
	We can further apply the bistellar flip on $\st(\{1, n-4, n\}, \Delta_{n-6})$ to obtain $\Delta_{n-5}$. The new sphere $\Delta_{n-5}$ has exactly one missing $3$-face, namely $\{2, n-3, n-2, n-1\}$. (It is not in $M_1$.) We can even apply the bistellar flip on $\st(1n, \Delta_{n-5})$, but the resulting complex is no longer neighborly.
\end{remark}

\begin{remark} \label{rem:Delta_i}
While we will not use the following result in this paper, it is worth remarking that Definition \ref{def:Delta_i} and Proposition \ref{prop:Delta_i} can be extended to higher dimensions. The simplest extension is as follows. For a fixed $k\geq 3$, take $\Delta^{2k}_1:=\partial C(2k+1,n)$, where $n$ is large enough, and then for $2\leq i\leq n-2k-2$, define  $\Delta^{2k}_i$ as the complex obtained from $\Delta^{2k}_{i-1}$ by applying the bistellar flip on $\st(\{1,3, \ldots, 2k-3\}\cup\{2k-3+i, n\}, \Delta^{2k}_{i-1})$. Then the complex $\Delta^{2k}_i$ is well-defined. Furthermore, it is a neighborly PL $2k$-sphere with $n$ vertices and $m_{k+1}={n-k-3 \choose k}-\sum_{1\leq \ell \leq i-1} (n-2k-1-\ell)$. The proof is very similar to that of Proposition \ref{prop:Delta_i}.

It is also not hard to see that $\Delta^{2k}_i$ is the boundary complex of the $(2k+1)$-ball 
$$D^k_i=\left(\overline{1n}*B(2k, [2,n-1])\right) \cup \left(\overline{\{1,2,n-1,n\}}* C_i\right).$$
Here $B(2k, [a, b])$ is the $(2k-1)$-ball generated by the following collection of facets of $\partial C(2k,n)$:
$$\{\{i_1, i_1+1, i_2, i_2+1, \ldots, i_k, i_k+1\} \ : \ a\leq i_1, \ i_k+1\leq b,\mbox{ and } \forall 1\leq j\leq k-1, \ i_j\leq i_{j+1}-2 \; \},$$
$C_1$ is the void complex, and for $i\geq 2$, $C_i$  is the stacked $(2k-2)$-ball whose set of facets is
$$\{\{3,4,\ldots,2k-2,2k-2+j-1, 2k-2+j\} \ : \ 2\leq j\leq i\}.$$
The ball $C_i$ belongs to the family of squeezed balls introduced in \cite{Kal}. More generally, one could consider complexes of the form $D^k(T):=\overline{1n}*B(2k, [2,n-1]) \cup \overline{\{1,2,n-1,n\}}*T$, where $T$ ranges over squeezed $(2k-3)$-balls with vertex set $[3,p]$ for some $p\leq n-2$. It would be interesting to determine for which squeezed balls $T$, $D^k(T)$ is a PL ball whose boundary is a neighborly sphere. It would also be interesting to determine the number of missing $(k+1)$-faces that $\partial D^k(T)$ has. 
\end{remark}

To prove Theorem \ref{thm: neighborly 4-spheres}, we need one more definition. We say that
a pure $(d-1)$-dimensional simplicial complex $\Delta$ is {\em shellable} if there is a linear order $F_1, F_2, \dots, F_k$ of the facets of $\Delta$ such that for all $2\leq i\leq k$, the subcomplex $\overline{F_i} \cap (\cup_{j<i} \overline{F_j})$ of $\partial\overline{F_i}$ is pure $(d-2)$-dimensional. Such order of facets is  called a {\em shelling order}. The unique minimal face of $\partial\overline{F_i}$ that is not in $\cup_{j<i} \overline{F_j}$ is called the {\em restriction face of $F_i$}. If $\Gamma$ is a PL $(d-1)$-sphere and $B\subset \Gamma$ is a shellable $(d-1)$-ball, then the complex obtained from $\Gamma$ by replacing $B$ with the cone over $\partial B$ is again a PL $(d-1)$-sphere.

{\smallskip\noindent {\it Proof of Theorem \ref{thm: neighborly 4-spheres}: \ }}
Fix $n\geq 8$. Recall that for $1\leq i\leq n-6$, $\Delta_i$ is a neighborly PL $4$-sphere with $n$ vertices and with $\binom{n-4-i}{2}$ missing $3$-faces. To construct a neighborly PL $4$-sphere with $n+1$ vertices and any value of $m_3$ in $\big\{1,2,\dots, \binom{n-4}{2}-2, \binom{n-4}{2}\big\}$, let $2\leq k\leq n-4$, and consider the complex $B_k$ generated by the facets 
\begin{eqnarray*} &&\{1,2,3,4,5\}, \{1,3,4,5,6\}, \{1,4,5,6,7\}, \dots, \{1, n-4,n-3,n-2,n-1\}, \\
	&&\{2,3,4,5,n\}, \{3,4,5,6,n\}, \dots,\{k,k+1, k+2, k+3, n\}.\end{eqnarray*}
The above ordering is a shelling of $B_k$; furthermore, each restriction face has size $\leq 2$. (The list of restriction faces consists of $\emptyset$, followed by vertices $6, 7, \dots, n$,  followed by edges $6n,7n,\dots,(k+3)n$, where the last part is empty if $k=2$.) Thus, $B_k$ is a $2$-stacked PL $4$-ball.

The ball $B_k$ is a subcomplex of $(\partial \overline{1n})*B(4, [2, n-1])$, which in turn is a subcomplex of $\Delta_i$. (Indeed, $(\partial \overline{1n})*B(4, [2, n-1])$ is a subcomplex of $\Delta_1$, and all the bistellar flips performed to get from $\Delta_1$ to  $\Delta_i$ only affected the open star of $1n$, so they did not touch this subcomplex.) Also, from the list of facets of $B_k$, we see that the minimal interior faces of $B_k$ are 
\begin{eqnarray*}&&\{3,4,5\}, \{4,5,6\}, \dots, \{k,k+1, k+2\}, \\
	&&\{1, k+1, k+2, k+3\}, \{1,k+2, k+3, k+4\}, \dots, \{1, n-4, n-3, n-2\},
\end{eqnarray*} if $k\geq 3$, and 
$$\{2,3,4,5\}, \{1,3,4,5\}, \{1,4,5,6\}, \dots, \{1,n-4, n-3,n-2\}$$ if $k=2$. That is, there are $n-4-k$ such $3$-faces if $k\geq 3$ and $n-5$ if $k=2$.

Let $\Gamma^k_i$ be obtained from $\Delta_i$ by replacing $B_k\subset \Delta_i$ with $\partial B_k*(n+1)$. Then $\Gamma^k_i$ is a neighborly PL $4$-sphere $\Gamma^k_i$ with $n+1$ vertices. (The $2$-neighborliness follows from the fact that $B_k$ is $2$-stacked and $V(B_k)=V(\Delta_i)$.) Since all missing $3$-faces of $\Delta_i$ contain both $1$ and $n$, they remain missing $3$-faces of $\Gamma^k_i$. Furthermore, the minimal interior faces of $B_k$ become ``new'' missing faces of $\Gamma^k_i$. Each other missing $3$-face of $\Gamma^k_i$ must be of the form $(n+1)\cup F$, where $F$ is  simultaneously a $2$-face of $\Delta_i$ and a {\em missing} $2$-face of 
of $B_k$. However no such $F$ exists. Indeed, if $F$ is a missing $2$-face of $B_k$, then the shelling of $B_k$ implies that $F$ contains a restriction edge, that is, $F=\{a <b <n\}$ for some $6\leq b\leq k+3$. Here $an$ is an edge, and so we see from the collection of facets of $B_k$ that $a\neq 1$. Since $ab$ is also an edge, it then follows that $b-a\leq 3$. This forces $F$ to be a $2$-face of $B_k$.

We conclude that $m_3(\Gamma^k_i)=m_3(\Delta_i)+(n-k-4)=\binom{n-4-i}{2}+(n-k-4)$ for $3\leq k\leq n-4$, and $m_3(\Gamma^2_i)=m_3(\Delta_i)+n-5=\binom{n-4-i}{2}+(n-5)$. Thus,
\begin{eqnarray*}
	\{m_3(\Gamma^k_i) : 1\leq i\leq n-6, 2\leq k\leq n-4\}&=&\left\{{m \choose 2}+s : 2\leq m\leq n-5, s\in [0,n-7]\cup\{n-5\}\right\} \\
	&=&\left\{1, 2, \dots, {n-4 \choose 2}-2\right\}\cup \left\{{n-4 \choose 2}\right\}.
\end{eqnarray*}
To complete the proof, it only remains to use Theorem \ref{thm: neighborly 2k-sphere m_k+1=0}, which asserts the existence of a  neighborly polytopal (and hence PL) $4$-sphere with $n+1\geq 9$ vertices and $m_3=0$.
\hfill$\square$\medskip

\begin{remark} \label{rem:5-polytopes-small}
	The cyclic polytope $C(5,7)$ is the only neighborly $5$-polytope with $7$ vertices; it has $m_3=1$. Looking at the list of all neighborly $5$-polytopes with $8$ vertices, one can check that the only possible values of $m_3$ are $1$ and $3$, that is, no such polytope has $m_3=2={3 \choose 2}-1$. Similarly, neighborly $5$-polytopes with $9$ vertices can have $m_3\in \{0,1,2,3,4,6\}$. (The list of all neighborly $5$-polytopes with $9$ vertices can be found in \cite{Finbow}.) At present, we do not know if there exist neighborly $4$-spheres with $n\geq 9$ vertices and $m_3={n-5 \choose 2}-1$. 
\end{remark}

In view of Corollary \ref{cor:characterization of m for 2-spheres}, Theorem \ref{thm: neighborly 4-spheres}, and Remark \ref{rem:5-polytopes-small} we posit the following conjecture.

\begin{conjecture} \label{conj:m_{k+1} not equal to max minus 1}
	Let $k\geq 2$ and let $\Delta$ be a neighborly $2k$-sphere with $n$ vertices. Then $m_{k+1}(\Delta)\neq \binom{n-k-3}{k}-1$. Furthermore, for $k\geq 3$, $n$ sufficiently large, and for any $0\leq m\leq \binom{n-k-3}{k}$ where $m\neq \binom{n-k-3}{k}-1$, there exists a neighborly $2k$-sphere with $n$ vertices and $m_{k+1}=m$.
\end{conjecture}

\begin{remark} As an additional evidence in support of the first part of Conjecture \ref{conj:m_{k+1} not equal to max minus 1},
	one can use Gale diagrams and arguments similar to those used in Section 7.2 below to show that for all $k\geq 2$, no neighborly $2k$-sphere with $n\leq 2k+4$ vertices has $m_{k+1}= \binom{n-k-3}{k}-1$.\footnote{By a result of Mani \cite{Mani1972}, all simplicial $2k$-spheres with  $\leq 2k+4$ vertices are polytopal.}  We omit the proof.
\end{remark}

\section{Neighborly $(2k+1)$-polytopes with $m_{k+1}=0$}
This section is devoted to proving Theorem \ref{thm: neighborly 2k-sphere m_k+1=0}.  Our proof consists of two parts.
First, for odd $k$ and $k=2$, we use Gale diagrams to construct  
neighborly $(2k+1)$-polytopes with few vertices and $m_{k+1}=0$ (see Section 7.2). We then recursively apply sewing to generate an infinite family of polytopes with the desired properties (see Section 7.3). We begin with a review of Gale diagrams (Section 7.1). We refer the reader to \cite{Gru-book,McMShep,Ziegler} for additional background on this fascinating topic.

\subsection{Gale diagrams}
Let $V=\{\mathbf{p}_1, \dots, \mathbf{p}_n\}$ be a set of points in $\R^d$ whose affine dimension is $d$. Let $D$ be the following $n\times (d+1)$ matrix
$$D=\begin{bmatrix}
	p_{1,1} & p_{1,2} & \dots & p_{1, d} & 1 \\
	p_{2,1} & p_{2,2} & \dots & p_{2, d} & 1 \\
	\dots & \dots & \dots & \dots & \dots \\
	p_{n,1} & p_{n,2} & \dots & p_{n, d} & 1 \\
\end{bmatrix},$$
where $\mathbf{p}_i=(p_{i,1}, \dots, p_{i, d})$. The space of affine dependences of $V$ has dimension $n-d-1$; let $\{\mathbf{a}_1, \dots, \mathbf{a}_{n-d-1}\}$ be a basis of this space. In particular, for each $\mathbf{a}_i=(a_{1,i}, \dots, a_{n,i})$, we have $\sum_{j=1}^n a_{j, i}\mathbf{p}_j=\mathbf{0}$ and $\sum_{j=1}^n a_{j, i}=0$. Let $\tilde{D}$ be the matrix whose column vectors are given by $\mathbf{a}_i^T$:
$$\tilde{D}=\begin{bmatrix}
	a_{1,1} & a_{1,2} & \dots & a_{1, n-d-1} \\
	a_{2,1} & a_{2,2} & \dots & a_{2, n-d-1} \\
	\dots & \dots & \dots & \dots  \\
	a_{n,1} & a_{n,2} & \dots & a_{n, n-d-1} \\
\end{bmatrix}.$$
Denote the $j$th row of $\tilde{D}$ by 
$\mathbf{\tilde{p}}_j=(a_{j,1}, \dots, a_{j, n-d-1})$. 

The (multi)set $\tilde{V}=\{\mathbf{\tilde{p}}_1, \dots,\mathbf{\tilde{p}}_n\}\subset \R^{n-d-1}$ is called the {\em Gale transform} of $V$. The \emph{Gale diagram} $\hat{V}$ of $V$ is defined by $\hat{V}=\{\mathbf{\hat{p}}_1, \dots, \mathbf{\hat{p}}_n\}$, where $\mathbf{\hat{p}}_i=\mathbf{0}$ if $\mathbf{\tilde{p}}_i=\mathbf{0}$ and $\mathbf{\hat{p}}_i=\mathbf{\tilde{p}}_i/\|\mathbf{\tilde{p}}_i\|$ otherwise. In particular, $\hat{V}$ is a subset of the unit $(n-d-2)$-sphere in $\R^{n-d-1}$ together with the origin. For $F\subset V$ we denote by $\tilde{F}$ and $\hat{F}$ the (multi)sets $\{\mathbf{\tilde{p}}_i: \mathbf{p}_i\in F\}$ and $\{\mathbf{\hat{p}}_i: \mathbf{p}_i\in F\}$, respectively.

Assume that $V$ is the vertex set of a $d$-polytope $P$. The main property of the Gale transforms and diagrams of polytopes (see \cite[Section 5.4]{Gru-book}) is that $F$ is the vertex set of a proper face of $P$ if and only if $\mathbf{0}\in \relint\conv(\tilde{V}\backslash \tilde{F})$, which happens if and only if $\mathbf{0}\in \relint\conv(\hat{V}\backslash \hat{F})$. 

Of a special interest to us is the case when $P$ is a {\em simplicial}  $d$-polytope and $|V|=d+3$. In this case, the origin is not in $\hat{V}$, and it is also not on any line segment connecting two points of $\hat{V}$. Hence, the Gale diagram of $P$ is a subset (possibly a multiset) of the unit circle $\Sp^1\subset \R^2$ centered at the origin, and no diameter of $\Sp^1$ that has a point of $\hat{V}$ at one of its ends can have a point of $\hat{V}$ at its other end. Such diameters of $\Sp^1$ are called the {\em diameters through $\hat{V}$}.  In our illustrations of the Gale diagram of $P$ we depict the points of $\hat{V}$ as black dots (with appropriate multiplicities) lying along $\Sp^1$ and we also draw the diameters through $\hat{V}$; see Figure \ref{figure 1} for an example. 

By \cite[Section 6.3]{Gru-book} and \cite[Section 3.3]{McMShep}, if $\hat{V}$ is the Gale diagram of $P$ as above, then applying the following two operations to $\hat{V}$ produces the Gale diagram of a polytope that is combinatorially equivalent to $P$. The two operations are: (1)  moving points of $\hat{V}$ along $\Sp^1$ as long as we do not alter the order of diameters through $\hat{V}$, and (2) merging two {\em adjacent} points of $\hat{V}$ together if they are not separated by any other diameter through $\hat{V}$. When applying the latter operation, the multiplicity of the resulting point of the Gale diagram is the sum of the multiplicities of the two merged points.

\subsection{Vertex-minimal constructions} 
What is the smallest number of vertices that a neighborly $2k$-sphere with $m_{k+1}=0$ can have (assuming such sphere exists)? It is known that any simplicial $(d-1)$-sphere with $d+2$ vertices is the join of the boundary complexes of two simplices whose dimensions add up to $d$.  Thus when $d=2k+1$, such a sphere $\Delta$ is neighborly if and only if it is of the form $\partial \sigma^k *\partial\sigma^{k+1}$, in which case $m_{k+1}(\Delta)=1$. It follows that any neighborly $2k$-sphere with $m_{k+1}=0$ must have at least $2k+4$ vertices.

We will now show that if $k$ is odd, that is, $k=2i-1\geq 3$, then there exists a neighborly $(2k+1)$-polytope with $m_{k+1}=0$ that has exactly $2k+4=4i+2$ vertices. (When $k=1$, the octahedron is a flag $3$-polytope with $2k+4=6$ vertices.)
 
\begin{definition}\label{Q_i}
	Let $k=2i-1\geq 1$. Consider a regular $(2i+1)$-gon inscribed in the unit circle with vertices labeled $z_0,z_1, z_{-1}\dots, z_i, z_{-i}$  in the order as in Figure \ref{figure 1} (that is, index $j$ corresponds to the angle $2\pi j/(2i+1)$ between $z_j$ and the positive direction of the $x$-axis). For each $j$, place two points, denoted $\mathbf{\hat{x}}_j$ and $\mathbf{\hat{y}}_j$, at vertex $z_j$. Let $Q_k$ be the $(2k+1)$-polytope with vertex set $V=\{\mathbf{x}_\ell, \mathbf{y}_\ell: -i\leq \ell \leq i, \,\ell \in \mathbb{Z}\}$ whose Gale diagram is given by 
	$C_k:=\hat{V}=\{\mathbf{\hat{x}}_\ell, \mathbf{\hat{y}}_\ell: -i\leq \ell \leq i, \,\ell \in \mathbb{Z}\}$.
	 
\end{definition} For instance, $Q_1$ is an octahedron (cf.~\cite[Figure 6.3.1]{Gru-book}), and the boundary complex of $Q_3$ is the vertex transitive triangulation $6\_10\_23\_1$ from \cite{manifold_page}.
\begin{figure}[ht]
	\centering
	\begin{tikzpicture}
		\draw[thick,red!90!black] (0,0) circle (3cm);
		
		\node[draw=none,minimum size=2cm,regular polygon,regular polygon sides=10] (a) {};
		\foreach \a in {1,2,...,5}
		{\filldraw[black] (\a*72-1.5: 3cm) circle (2pt);
			\filldraw[black] (\a*72+1.5: 3cm) circle (2pt);
			\filldraw[black] (289.5: 3cm) circle (2pt);
			\filldraw[black] (286.5: 3cm) circle (2pt);
			\draw (0: 3.6cm) node{$\mathbf{\hat{x}}_{0}, \mathbf{\hat{y}}_{0}$};
			\draw (72: 3.25cm) node{$\mathbf{\hat{x}}_{1}, \mathbf{\hat{y}}_{1}$};
			\draw (144: 3.5cm) node{$\mathbf{\hat{x}}_{2}, \mathbf{\hat{y}}_{2}$};
			\draw (216: 3.5cm) node{$\mathbf{\hat{x}}_{-2}, \mathbf{\hat{y}}_{-2}$};
			\draw (288: 3.3cm) node{$\mathbf{\hat{x}}_{-1}, \mathbf{\hat{y}}_{-1}$};
			\draw (\a*72: 3cm)--(\a*72+180: 3cm);}
	\end{tikzpicture}
	\caption{The Gale diagram of $Q_3$}\label{figure 1}
\end{figure}
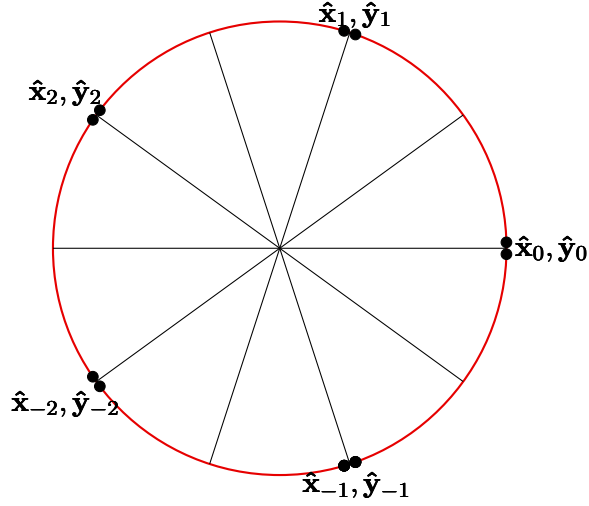

\begin{proposition} \label{prop:neighborliness-of-Q_i}
	The polytope $Q_k$ is a simplicial $(2k+1)$-polytope with $2k+4$ vertices; it is neighborly and all of its missing faces have dimension $k$. 
\end{proposition}	

\begin{proof}
	That $Q_k$ is simplicial and neighborly follows easily from the Gale diagram: $Q_k$ is simplicial because the origin is not contained in the relative interior of the convex hull of any two elements of $C_k$, and $Q_k$ is neighborly because every open semicircle contains at least $k+1=2i$ elements of $C_k$; see \cite[Exercise 7.3.7]{Gru-book}. 
	
	To complete the proof, it suffices to show that $F$ is a missing face of $Q_k$ if and only if $\hat{F}$ consists of $i$ consecutive double points of the Gale diagram (i.e., $\{\mathbf{\hat{x}}_1,\mathbf{\hat{y}}_1,\dots,\mathbf{\hat{x}}_{i},\mathbf{\hat{y}}_{i}\}$ and all rotations of this set); in particular, each missing face has dimension $2i-1=k$. First, it is immediate from the Gale diagram that if $\hat{F}=\{\mathbf{\hat{x}}_1,\mathbf{\hat{y}}_1,\dots,\mathbf{\hat{x}}_{i},\mathbf{\hat{y}}_{i}\}$, then $F$ is a missing face. By symmetry, this also holds for all rotations of $\{\mathbf{\hat{x}}_1,\mathbf{\hat{y}}_1,\dots,\mathbf{\hat{x}}_{i},\mathbf{\hat{y}}_{i}\}$. 
	
	For the other direction, we claim that if $\hat{G}\subset C_k$ has size $k+2=2i+1$ and $\hat{G}$ does not contain a rotation of $\{\mathbf{\hat{x}}_1,\mathbf{\hat{y}}_1,\dots,\mathbf{\hat{x}}_{i},\mathbf{\hat{y}}_{i}\}$, then $G$ is the vertex set of a face. Indeed, 
	deleting such $\hat{G}$ from $C_k$ destroys at most $i$ vertices of the regular $(2i+1)$-gon; furthermore, if it destroys exactly $i$ vertices, then the remaining $i+1$ vertices of the $(2i+1)$-gon do not form a consecutive block. It follows that the origin lies in the interior of $\conv(C_k\backslash\hat{G})$, completing the proof. 
\end{proof}

Before we proceed, we discuss additional properties of $Q_k$ that will be crucial for the inductive procedure in Section 7.3. Note that for any $0\leq \ell\leq i$, the antipode of $z_\ell$ on the unit circle  lies between $z_{-(i-\ell)}$ and $z_{-(i-\ell+1)}$  (here we identify $z_{-(i+1)}$ with $z_i$). For this reason, we refer  to each of $z_{-(i-\ell)}$ and $z_{-(i-\ell+1)}$ as an {\em almost antipodal point} of $z_\ell$; we also refer to the corresponding points of $C_k$ as {\em almost antipodal points} of $\mathbf{\hat{x}}_\ell$ and $\mathbf{\hat{y}}_\ell$.

We now define the following sequence of pairwise disjoint edges $e_1, e_2, \dots, e_{k}$ of $Q_k$: let
$$e_1=\mathbf{x}_0\mathbf{x}_i,\quad e_2=\mathbf{y}_0\mathbf{x}_{-i},$$
$$\text{and   }  e_{4j-1}=\mathbf{x}_{-j}\mathbf{y}_{i-j+1},\; e_{4j}=\mathbf{x}_j\mathbf{y}_{-(i-j+1)},\; 
e_{4j+1}=\mathbf{y}_{-j}\mathbf{x}_{i-j}, \; e_{4j+2}=\mathbf{y}_j\mathbf{x}_{-(i-j)} \text{  for $j\geq 1$}.$$
In particular,  the vertices of each edge in the sequence correspond to almost antipodal points of $C_k$. Furthermore, the points of the Gale diagram that correspond to the vertices of $e_{2\ell-1}$  are symmeric about the horizontal diameter of the unit circle to the points corresponding to the vertices of $e_{2\ell}$.

\begin{lemma} \label{lem:links-of-Q_i}
	Consider the boundary complex of $Q_k$, $\partial Q_k$. All links are computed in this complex.
	\begin{enumerate}
		\item The link of any edge $\mathbf{vw}$, where $\mathbf{\hat{v}},\mathbf{\hat{w}}\in C_k$ are almost antipodal, is a neighborly $(2k-2)$-sphere.
		\item For $1\leq j\leq k$, $F_j=e_1\cup \dots\cup e_j$ is a face of $\partial Q_k$. Furthermore, for $1\leq j\leq k-1$, the link of $F_j$ is a neighborly $(2k-2j)$-sphere on vertex set $V(Q_k)\backslash F_j$, and it is isomorphic to $\partial Q_{k-j}$ if $j$ is even.
		\item In particular, the link of $e_1\cup \dots \cup e_{k-1}$ is the octahedral $2$-sphere.
	\end{enumerate}
\end{lemma}
\begin{proof} 
	For part 1, assume without loss of generality that $e=\mathbf{x}_0\mathbf{x}_i$. To show that $\lk(e)$ is $(k-1)$-neighborly, it suffices to check that $\mathbf{0} \in \intr(\conv(W))$ for any $(k+3)$-subset $W$ of $C_k\backslash \{\mathbf{\hat{x}}_0,\mathbf{\hat{x}}_i\}$. Since $k+3=2i+2$, it follows that the points of $W$ cover at least $i+1$ vertices of the regular $(2i+1)$-gon. Hence $\mathbf{0}\in\intr(\conv(W))$ unless $W$ consists of $i+1$ consecutive double points. However, the latter is impossible because either $W$ is a subset of $C_k\backslash \{\mathbf{\hat{x}}_0, \mathbf{\hat{y}}_0, \mathbf{\hat{x}}_i, \mathbf{\hat{y}}_i\}$, or $W$ contains at least one single point ($\mathbf{\hat{y}}_0$ or $\mathbf{\hat{y}}_i$).
	
	For part 2, we prove that the Gale diagram $C'$ obtained from $C_k$ by removing $\{\mathbf{\hat{x}}_0, \mathbf{\hat{y}}_0, \mathbf{\hat{x}}_i, \mathbf{\hat{x}}_{-i}\}$ is equivalent to the Gale diagram with double points positioned at the vertices of the regular $(2i-1)$-gon. Indeed, the antipodes of both $\mathbf{\hat{y}}_i$ and $\mathbf{\hat{y}}_{-i}$ lie between $\mathbf{\hat{x}}_1$ and $\mathbf{\hat{x}}_{-1}$. Thus, in $C'$, $\mathbf{\hat{y}}_i$ and $\mathbf{\hat{y}}_{-i}$ are adjacent points that are not separated by any diameter through $C'$. By the discussion at the end of Section 7.1, we can merge these two points to form a double point lying in the position opposite to $\mathbf{\hat{x}}_0$. Furthermore, since for $1\leq \ell\leq i-1$, the antipode of	$\mathbf{\hat{x}}_\ell$ (resp.~$\mathbf{\hat{x}}_{-\ell}$) lies between $\mathbf{\hat{x}}_{-(i-\ell)}$ and $\mathbf{\hat{x}}_{-(i-\ell+1)}$ (resp.~$\mathbf{\hat{x}}_{i-\ell}$ and $\mathbf{\hat{x}}_{i-\ell+1}$), we can then move the other $2i-2$ double points along the circle respecting the order of the corresponding diameters, so that the resulting configuration consists of $2i-1$ double points positioned at the vertices of the regular $(2i-1)$-gon.\footnote{To achieve this, first move the double points $\mathbf{\hat{x}}_1,\mathbf{\hat{y}}_1$ and  $\mathbf{\hat{x}}_{-1},\mathbf{\hat{y}}_{-1}$, then move the double points $\mathbf{\hat{x}}_{i-1},\mathbf{\hat{y}}_{i-1}$ and $\mathbf{\hat{x}}_{-(i-1)}, \mathbf{\hat{y}}_{-(i-1)}$, followed by  $\mathbf{\hat{x}}_2,\mathbf{\hat{y}}_2$ and $\mathbf{\hat{x}}_{-2},\mathbf{\hat{y}}_{-2}$, then $\mathbf{\hat{x}}_{i-2},\mathbf{\hat{y}}_{i-2}$ and $\mathbf{\hat{x}}_{-(i-2)}, \mathbf{\hat{y}}_{-(i-2)}$, etc.} (See Figure \ref{equivalent Gale diagram} for an illustration in the case of $k=3$.) In particular, this means that $\lk(e_1\cup e_2)$ is isomorphic to $\partial Q_{k-2}$. Furthermore, the points of the resulting Gale diagram that correspond to the vertices of $e_t$ for any $t\geq 3$ are almost antipodal points of the regular $(2i-1)$-gon. The statement of part 2 then follows from part 1 and Proposition \ref{prop:neighborliness-of-Q_i} by induction on $j$.
	
	According to part 2, the link of $e_1\cup\dots\cup e_{k-1}$ is combinatorially isomorphic to $\partial Q_1$. Part 3 follows because $Q_1$ is an octahedron.\end{proof}

\begin{figure}[ht]
	\centering
	\begin{tikzpicture}
		\draw[thick,red!90!black] (0,0) circle (3cm);
		
		\node[draw=none,minimum size=2cm,regular polygon,regular polygon sides=10] (a) {};
		\foreach \a in {1,2,...,4}
		{\filldraw[black] (70.5: 3cm) circle (2pt);
			\filldraw[black] (73.5: 3cm) circle (2pt);
			\filldraw[black] (144: 3cm) circle (2pt);
			\filldraw[black] (216: 3cm) circle (2pt);
			\filldraw[black] (289.5: 3cm) circle (2pt);
			\filldraw[black] (286.5: 3cm) circle (2pt);
			\draw (72: 3.25cm) node{$\mathbf{\hat{x}}_{1}, \mathbf{\hat{y}}_{1}$};
			\draw (144: 3.3cm) node{$\mathbf{\hat{y}}_{2}$};
			\draw (216: 3.4cm) node{$\mathbf{\hat{y}}_{-2}$};
			\draw (288: 3.3cm) node{$\mathbf{\hat{x}}_{-1}, \mathbf{\hat{y}}_{-1}$};
			\draw (\a*72: 3cm)--(\a*72+180: 3cm);}
	\end{tikzpicture}
	\hspace{1cm}
	\begin{tikzpicture}
		\draw[thick,red!90!black] (0,0) circle (3cm);
		
		\node[draw=none,minimum size=2cm,regular polygon,regular polygon sides=10] (a) {};
		\foreach \a in {1,2,3}
		{\filldraw[black] (\a*120-58.5: 3cm) circle (2pt);
			\filldraw[black] (\a*120-61.5: 3cm) circle (2pt);
			\draw (60: 3.3cm) node{$\mathbf{\hat{x}}_{1}, \mathbf{\hat{y}}_{1}$};
			\draw (180: 3.7cm) node{$\mathbf{\hat{y}}_{2}, \mathbf{\hat{y}}_{-2}$};
			\draw (300: 3.3cm) node{$\mathbf{\hat{x}}_{-1}, \mathbf{\hat{y}}_{-1}$};
			\draw (\a*120-60: 3cm)--(\a*120+120: 3cm);}
	\end{tikzpicture}
	\caption{The Gale diagram $C_3\backslash (\hat{e}_1\cup \hat{e}_2)$ and an equivalent Gale diagram}\label{equivalent Gale diagram}
\end{figure}
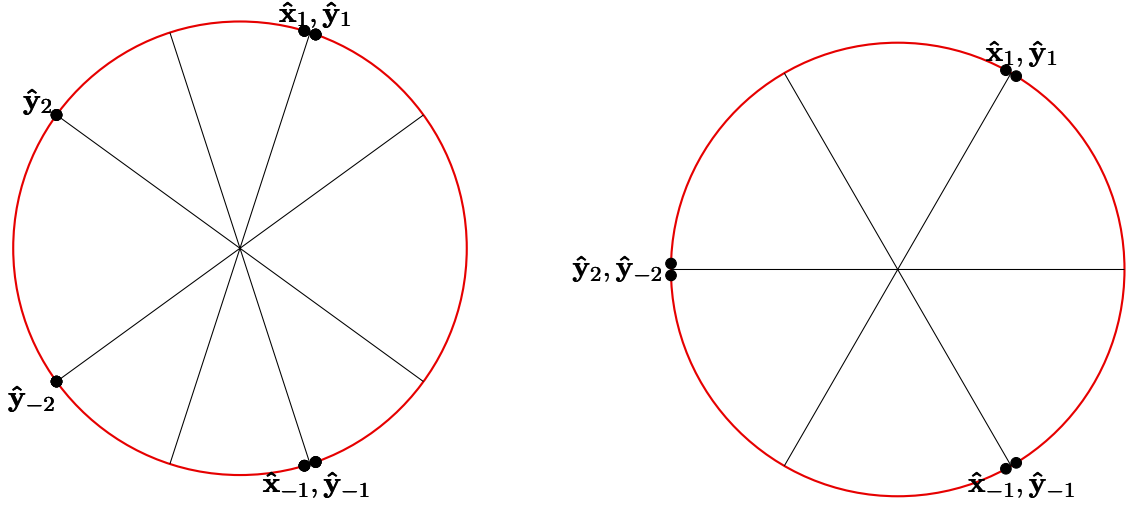

We will now show that when  $k$ is odd, $\partial Q_k$ is the only neighborly $2k$-sphere with $f_0=2k+4$ and $m_{k+1}=0$ while when $k$ is even, a neighborly $2k$-sphere with $m_{k+1}=0$ must have at least $2k+5$ vertices.

\begin{proposition}\label{Prop: vertex-minimal}
	Let $\Delta$ be a neighborly $2k$-sphere with $n$ vertices, and assume that all missing faces of $\Delta$ have dimension $k$. Then $n\geq 2k+4$ if $k$ is odd and $n \geq 2k+5$ if $k$ is even. Furthermore, if $k$ is odd and $n=2k+4$, then $\Delta$ is isomorphic to $\partial Q_k$.
\end{proposition}
\begin{proof}
	In the beginning of Section 7.2, we saw that $n\geq 2k+4$. Thus, assume that $\Delta$ has $2k+4$ vertices, and hence, by a result of Mani \cite{Mani1972}, it is the boundary complex of some polytope $P$. Let $C$ be the Gale diagram of $P$. Then neighborliness of $\Delta$ guarantees that every diameter through $C$ has at least $k+1$ elements of $C$ (counted with multiplicities) on each of its open sides. In particular, no point of the diagram can have multiplicity larger than two.

	Let $u_0\in C$ be a single point. Then the diameter through $u_0$ has $k+1$ elements of $C$ on one open side, and $k+2$ on the other; denote them by $u_1,\dots ,u_{k+2}$ according to their distances from $u_0$ with $u_1$ being the closest. The vertices corresponding to $u_1,\dots,u_{k+2}$ do not form a face of $P$. Since there are no missing faces of size $k+2$, there is some $1\leq j\leq k+2$ such that $\mathbf{0}\notin\relint(\conv((C\backslash \{u_1, \dots u_{k+2}\})\cup u_j)$. Hence the shorter arc from $-u_0$ to $-u_1$ in $\Sp^1$ contains no elements of $C$. Thus, as was explained in Section 7.1, merging $u_1$ with $u_0$ does not change the combinatorial type of $P$. 
	
	Applying the same argument to other single points of $C$, we conclude that $C$ is equivalent to a Gale diagram where no point is single, and hence all points are double (as points of multiplicity larger than two violate neighborliness).
Furthermore, for every two adjacent double points $v$ and $v'$, there must be a point on the shorter arc from $-v$ to $-v'$, or else we would be able to merge $v$ and $v'$ creating a point of multiplicity larger than two. Thus, every diameter through $C$ 
contains exactly $k+1$ elements on each of its open sides, and they are presented by $(k+1)/2$ double points. This is impossible if $k$ is even. Therefore, when $k$ is even, $\Delta$ must have at least $2k+5$ vertices. Finally, if $k$ is odd, then the above description of the Gale diagram of $P$ shows that it is equivalent to that of $Q_k$,which implies that $\Delta=\partial Q_k$.
\end{proof}

To close this section, we provide a vertex-minimal construction of a neighborly $5$-polytope with $m_3=0$.

\begin{figure}[ht]
	\centering
	\begin{tikzpicture}[scale=0.8]
		\draw (4, 0)node[right]{8}--(-4,0)node[left]{6}--(0,6.9)node[above]{3}--(4,0)--(-1.7,3.3);
		\draw (0, 0.5)node[below]{1}--(1.7,3.3)node[right]{2}--(-1.7, 3.3)node[left]{4}--(0,0.5)--(-4,0);
		\draw (1.7,3.3)--(-4,0)--(-1.7,3.3)--(0,6.9)--(0,0.5)--(4,0)--(1.7,3.3)--(0,6.9);
		\draw[red] (-2,0.6)node[left]{9}--(2.5,1.4)node[below]{5}--(-0.4,4.8)node[left]{7}--(-2, 0.6);
	\end{tikzpicture}
	\caption{The affine Gale diagram of $\mathcal{P}^0_{42}$}
	\label{AffineGaleDiagram}
\end{figure}
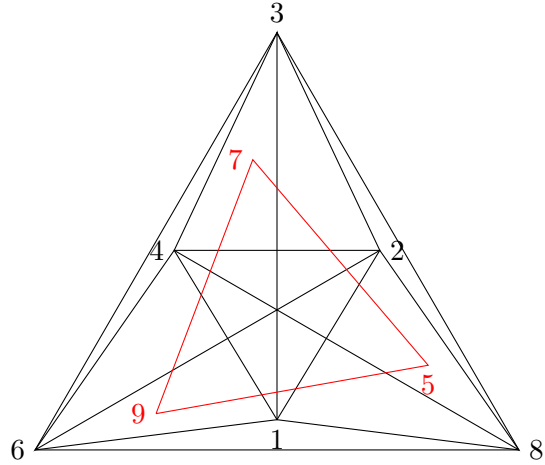 
\begin{definition}
	Consider the simplicial complex generated by the facets
	\begin{equation}\label{list of facets}
		\begin{split}
			&\{1,3,6,8,9\},\{1,3,4,8,9\},\{3,4,6,8,9\},\{2,4,6,8,9\},\{1,3,5,6,8\},\{1,2,3,5,6\},\\
			&\{2,3,5,6,8\},\{1,3,4,5,8\},\{2,3,6,7,8\},\{2,4,6,7,8\},\{3,4,6,7,8\},\{1,2,3,6,7\},\\
			&\{1,5,6,8,9\},\{2,3,5,7,8\},\{3,4,6,7,9\},\{1,2,4,5,9\},\{1,2,4,5,7\},\{1,2,4,7,9\},\\
			&\{2,5,6,8,9\},\{2,4,5,8,9\},\{2,4,5,7,8\},\{2,4,6,7,9\},\{1,3,6,7,9\},\{1,2,6,7,9\},\\
			&\{1,2,5,6,9\},\{1,2,3,5,7\},\{3,4,5,7,8\},\{1,3,4,7,9\},\{1,3,4,5,7\},\{1,4,5,8,9\}.
		\end{split}
	\end{equation}
It is the boundary complex of the $5$-polytope $\mathcal{P}^0_{42}$ from \cite{Finbow}.
\end{definition}

One can ``picture'' $\mathcal{P}^0_{42}$ in $\R^2$ using the notion of  the {\em affine Gale diagram}. We refer the reader to \cite[Section 6.4]{Ziegler} for precise definitions, and merely mention that while the Gale diagram of $\mathcal{P}^0_{42}$ lives in $\R^3$, the affine Gale diagram of $\mathcal{P}^0_{42}$ lives in $\R^2$. As in the case of the usual Gale diagram, the affine Gale diagram consists of nine points corresponding to the vertices of $\mathcal{P}^0_{42}$. The difference is that in the affine case, each point is 
colored red or black. For simplicity, we label these points using the labels of the corresponding vertices of $\partial \mathcal{P}^0_{42}$.  A set $F\subset [9]$ corresponds to the vertex set of a facet if and only if the red and black points of $F^c:=[9]\backslash F$ form a {\em Radon partition}, i.e., the convex hull of the red points of $F^c$ intersects the convex hull of the black points of $F^c$. 

The affine Gale diagram of $\mathcal{P}^0_{42}$ is shown in Figure \ref{AffineGaleDiagram}. Here $V_1=\{5,7,9\}$ is the set of red points and $V_2=\{1,2,3,4,6,8\}$ is the set of black points. We say that a Radon partition $(W_1, W_2)$, where $W_i\subset V_i$, is of type $(j, 4-j)$ if $|W_1|=j$. Each red edge in the picture crosses $4$ black edges, and hence there are $12$ Radon partitions of type $(2,2)$. Their complements give the first $12$ facets in (\ref{list of facets}). Similarly, the complements of $18$ Radon partitions of type $(1,3)$ give the remaining $18$ facets in (\ref{list of facets}).  The missing $2$-faces of $\mathcal{P}^0_{42}$ are
$$\{9, 2, 3\}, \{2, 3, 4\}, \{9, 5, 7\}, \{5, 6, 7\}, \{8, 1, 7\}, \{8, 1, 2\}, \{9, 3, 5\}, \{4, 5, 6\}, \{8, 9, 7\}, \{1, 4, 6\},$$
and one can check that there are no missing $3$-faces.

By the enumeration of neighborly $5$-polytopes with $9$ vertices in \cite{Finbow}, $\mathcal{P}^0_{42}$ is the only vertex-minimal neighborly $5$-polytope all of whose missing faces have dimension $2$. We do not know if there exist other neighborly (non-polytopal) $4$-spheres with $9$ vertices and $m_3=0$. The complex $4\_10\_1\_1$ from \cite{manifold_page} is a vertex-transitive neighborly $4$-sphere with $10$ vertices and $m_3=0$.

\subsection{Generating infinite families}
We now discuss an inductive procedure which, using the vertex-minimal neighborly spheres from Section 7.2 as the base case, will allow us to construct infinite families of neighborly $2k$-spheres all of whose missing faces have dimension $k$. 

Our inductive procedure relies on a few lemmas. We say that a $d$-ball is {\em exactly $i$-stacked} if all of its minimal interior faces are of dimension $d-i$. For example, a $d$-simplex is exactly $0$-stacked and a stacked $d$-ball that is not a $d$-simplex is exactly $1$-stacked. If $B\subset S$ are two simplicial complexes, then we say that $B$ is {\em induced on its $k$-skeleton in $S$} if every missing face of $B$ of dimension $\geq k+1$ is also a missing face of $S$. Throughout this section, we assume that $k\geq 1$.

\begin{lemma} \label{Claim 1}
	Let $\Gamma$ be a $k$-neighborly PL $2k$-sphere with $V(\Gamma)=[n]$.  
	Assume that $\Gamma$ contains a  PL $2k$-ball $D$ with the following properties:
	\begin{enumerate}
		\item $D$ is $(k-1)$-neighborly with $V(D)=[n]$ (this condition is omitted if $k=1$),
		\item $D$ is exactly $k$-stacked, 
		\item $D$ is induced on its $(k-1)$-skeleton in $\Gamma$.
	\end{enumerate} 
	Let $\Gamma'$ be the complex obtained from $\Gamma$ by replacing $D$ with $\partial D*(n+1)$.
	Then $\Gamma'$ is a $k$-neighborly PL $2k$-sphere with $V(\Gamma')=[n+1]$. Furthermore, if all missing faces of $\Gamma$ have dimension $k$, then so do all missing faces of $\Gamma'$.
\end{lemma}

\begin{proof}
	Since $\Gamma$ is a PL $2k$-sphere and $D$ is a PL $2k$-ball, it follows that $\Gamma'$ is a PL $2k$-sphere. Let $M$ be the set of minimal interior faces of $D$. By condition 2, each element of $M$ is of dimension $k$. Hence $\skel_{k-1}(\partial D)=\skel_{k-1}(D)$. This implies that $V(\Gamma')=[n+1]$; in particular, if $k=1$, then $\Gamma'$ is $k$-neighborly w.r.t.~$[n+1]$. Furthermore, for $k>1$, the fact that $\skel_{k-1}(\partial D)=\skel_{k-1}(D)$ together with the assumptions that  $D$ is $(k-1)$-neighborly (condition 1) and $\Gamma$ is $k$-neighborly implies that $\Gamma'$ is $k$-neighborly.
	
	Finally, assume that all missing faces of $\Gamma$ have dimension $k$. Note that a missing face of $\Gamma'$ is either a missing face of $\Gamma$, or a minimal interior face of $D$, or a missing face containing vertex $n+1$. In the former two cases, it must be of dimension $k$ by our assumptions on $\Gamma$ and $D$. In the latter case, it must be of the form $F\cup(n+1)$, where $F$ is a face of $\Gamma$ but a missing face of $D$. Since $D$ is induced on its $(k-1)$-skeleton in $\Gamma$ (condition 3), it follows that $\dim F\leq k-1$, or equivalently, that $|F|\leq k$. On the other hand, the assumption that $D$ is $(k-1)$-neighborly implies that $|F|=k$. (If $k=1$, then $|F|=1$ since the empty face cannot be a missing face.) Hence the missing face $F\cup(n+1)$ has dimension $k$.
\end{proof}

When $D\subset \Gamma$ are pure complexes of the same dimension, we denote by $\Gamma\backslash D$ the subcomplex of $\Gamma$ generated by the facets of $\Gamma$ that do not belong to $D$. We call $\Gamma\backslash D$ the {\em complement of $D$ in $\Gamma$}.

\begin{lemma} \label{Claim 2}
	Let  $\Gamma$ be a $k$-neighborly PL $2k$-sphere. Let $D$ be a PL $2k$-ball in $\Gamma$ such that
	\begin{enumerate}
		\item $D$ is $(k-1)$-neighborly with $V(D)=V(\Gamma)$ (this condition is omitted if $k=1$),
		\item $D$ is exactly $k$-stacked, 
		\item $D$ is induced on its $(k-1)$-skeleton in $\Gamma$.
	\end{enumerate} Then the complement $B$ of $D$ in $\Gamma$ is a $k$-neighborly PL $2k$-ball with $V(B)=V(\Gamma)$. Furthermore, $B$ is exactly $(k+1)$-stacked and induced on its $k$-skeleton in $\Gamma$.
\end{lemma}

\begin{proof}
	Since $D$ is exactly $k$-stacked (condition 2), the minimal interior faces of $D$ are of dimension $2k-k=k$, and so $V(B)=V(\Gamma)$. Furthermore, since $\Gamma$ is  $k$-neighborly, it follows that the complement $B$ of $D$ is also $k$-neighborly. Let $F$ be a minimal interior face of $B$, or equivalently, a face of $\Gamma$ that is a missing face of $D$. Since $D$ is $(k-1)$-neighborly (condition 1), the dimension of $F$ is at least $k-1$. (When $k=1$, this holds because the empty face is not a missing face.) As $D$ is induced on its $(k-1)$-skeleton in $\Gamma$ (condition 3), we conclude that $\dim F \leq k-1$. Thus $\dim F =k-1=2k-(k+1)$, and so $B$ is exactly $(k+1)$-stacked. Finally, we show that $B$ is induced on its $k$-skeleton in $\Gamma$. Let $F$ be a missing face of $B$ of dimension $\geq k+1$. Then $F$ is either a missing face of $\Gamma$, in which case we are done, or $F$ is a minimal interior face of $D$, in which case it can only be of dimension $k-1$. Thus,  no face of $\Gamma$ dimension $\geq k+1$ is a missing face of $B$. 
\end{proof}

\begin{lemma}\label{Claim 3}
	Let $\Sigma$ be a $k$-neighborly PL $2k$-sphere. 
	Let $E=\{e_1, e_2, \dots, e_k\}$ be a sequence of pairwise disjoint edges of $\Sigma$ such that for all $0\leq j\leq k$, $F_j=e_1\cup \dots \cup e_j$ is a face of $\Sigma$,  and let $\Gamma_{k-j}:=\lk(F_j,\Sigma)$. Assume further that for all $0\leq j\leq k-1$, $\Gamma_{k-j}$ satisfies the following conditions: 

\begin{itemize}
	\item[$(*)$] $\Gamma_{k-j}$ is a $(k-j)$-neighborly $(2k-2j)$-sphere with vertex set $V(\Sigma)\backslash F_j$ (this condition is omitted if $j=k-1$);
	\item[$(**)$] if $k-j$ is odd, then all missing faces of $\Gamma_{k-j}$ have dimension $k-j$.
\end{itemize}
	
	Define the following collection of balls $(D_j, B_j)$ in $\Gamma_j$ for $1 \leq j\leq k$:
	$$D_1=\overline{e_k}*\Gamma_0 \text{ and } B_1=\Gamma_1\backslash D_1,$$
	$$\text{and for } 2\leq j\leq k, \quad D_j=\overline{e_{k+1-j}}*B_{j-1} \text{ and } B_j=\Gamma_j\backslash D_j .$$
	Then  for all $1\leq j\leq k$, 
	\begin{enumerate}
		\item $D_j$ is $(j-1)$-neighborly with $V(D_j)=V(\Gamma_j)$ (if $j>1$), exactly $j$-stacked, and induced on its $(j-1)$-skeleton in $\Gamma_j$.
		\item $B_j$ is $j$-neighborly with $V(B_j)=V(\Gamma_j)$, exactly $(j+1)$-stacked, and induced on its $j$-skeleton in $\Gamma_j$.
	\end{enumerate}
\end{lemma}
\begin{proof}
Observe that $F_0=\emptyset$ and $\Gamma_k=\Sigma$. By Lemma \ref{Claim 2}, if $D_j$ satisfies the desired properties, then by $(*)$, 
 so does $B_j$. To prove the claim about $D_j$, we induct on $j$. In the base case of $j=0$, $\Gamma_0$ is a $0$-sphere; assume its vertices are $a$ and $b$. By $(**)$,
  $\Gamma_1$ is flag. So if $ab$ is an edge of $\Gamma_1$, then the $2$-sphere $\Gamma_1$ contains the $3$-simplex on $V(D_1)$, which is impossible. Thus, $D_1$ is exactly $1$-stacked and induced on its $0$-skeleton in $\Gamma_1$. 
	
	For our inductive step, assume that $j\geq 2$ and that $D_{j-1}$ and $B_{j-1}$ satisfy the desired conditions. Since $D_j=\overline{e_{k-j+1}}*B_{j-1}$, the assumptions that $B_{j-1}$ is $(j-1)$-neighborly and exactly $j$-stacked, imply that so is $D_j$. Furthermore, that $V(B_{j-1})=V(\Gamma_{j-1})$ implies that $V(D_j)=V(\Gamma_j)$. To see that $D_j$ is induced on its $(j-1)$-skeleton in $\Gamma_j$, let $F$ be a missing face of $D_j$ of dimension $\geq j$. Then $F$ is a missing face of $B_{j-1}$ of dimension $\geq j$. But $B_{j-1}$ is induced on its $(j-1)$-skeleton in $\Gamma_{j-1}$, and so $F$ must be a missing face of $\Gamma_{j-1}$. If $j-1$ is odd, then by $(**)$, $\Gamma_{j-1}$ has no missing faces of dimension $\geq j$. We conclude that, in this case, all missing faces of $D_j$ have dimension $\leq j-1$, and so $D_j$ is induced on its $(j-1)$-skeleton in $\Gamma_j$. Finally, in the case that $j$ is odd, since $\Gamma_{j-1}=\lk(e_j, \Gamma_j)$, there must be a subset $X$ of $e_j$ such that $F\cup X$ is a missing face of $\Gamma_j$. In particular, $\dim (F\cup X)\geq \dim(F)\geq j$. But 
	since $j$ is odd, by $(**)$, all missing faces of $\Gamma_j$ have dimension $j$. This implies that $X=\emptyset$ and that $F$ is a missing face of $\Gamma_j$ (of dimension $j$). Thus, we again conclude that $D_j$ is induced on its $(j-1)$-skeleton in $\Gamma_j$. This completes the proof.  
\end{proof}

\begin{lemma}\label{Claim 4}
Let $\Sigma$ be a $2k$-sphere with $V(\Sigma)=[n]$. Assume that the pair $(\Sigma, E=\{e_1, e_2, \dots, e_k\})$ satisfies all the assumptions of Lemma \ref{Claim 3} and let $D_k$ be defined as in that lemma; further, by relabeling the vertices, if necessary, assume that $e_j=\{n+1-2j, n+2-2j\}$ for all $1\leq j\leq k$. Let $e'_j=\{n+2-2j,n+3-2j\}$ for all $1\leq j\leq k$, and let the complex $\Sigma'$ be obtained from $\Sigma$ by sewing a new vertex $n+1$ on $D_k$, i.e., by replacing $D_k$ with $\partial D_k*(n+1)$. Then $V(\Sigma')=[n+1]$ and the pair $(\Sigma', E':=\{e'_1, e'_2,\ldots, e'_k\})$ satisfies all the assumptions of Lemma \ref{Claim 3}. 
\end{lemma}
\begin{proof}
	Let $F'_j=e_1'\cup \dots \cup e_j'$ and $\Gamma_{k-j}'=\lk(F_j', \Sigma')$. We need to check that $\Gamma_{k-j}'$ satisfies conditions $(*)$ and $(**)$ of Lemma \ref{Claim 3}. We start with $j=0$. By Lemma \ref{Claim 3}, $D_k$ is $(k-1)$-neighborly with $V(D_k)=V(\Sigma)$ (if $k>1$), exactly $k$-stacked, and induced on its $(k-1)$-skeleton in $\Sigma$. Hence by Lemma \ref{Claim 1}, $\Gamma'_k=\Sigma'$ satisfies $(*)$ and $(**)$. 
	
	We continue to follow the notation of Lemma \ref{Claim 3}. When $j=1$, using that $e_1'=\{n, n+1\}$ and $\lk(n+1, \Sigma')=\partial D_k$, we see that
	\begin{equation*}
		\begin{split}
			\Gamma_{k-1}'&=\lk(e_1', \Sigma')=\lk(n, \partial D_k)=\lk(n, \partial (\overline{(n-1)n}*B_{k-1}))\\
			&=B_{k-1}\cup (\partial B_{k-1}*(n-1))=(\Gamma_{k-1}\backslash D_{k-1})\cup (\partial D_{k-1}*(n-1)).
		\end{split}
	\end{equation*}
	Thus, 	$\Gamma'_{k-1}$ is obtained from $\Gamma_{k-1}$ by sewing on $D_{k-1}$ and $V(\Gamma'_{k-1})=[n-1]$. Invoking Lemma \ref{Claim 1} once again we obtain that $\Gamma_{k-1}'$ satisfies $(*)$ and $(**)$. Induction on $j$, with the above argument serving as the induction step, finishes the proof.
\end{proof}

The above lemmas lead to the promised inductive procedure.

\begin{corollary} \label{main-cor}
	Assume $(\Sigma, E=\{e_1,\ldots,e_k\})$ satisfies all the assumptions of Lemma \ref{Claim 3}. Assume also that all missing faces of $\Sigma$ have dimension $k$.\footnote{If $k$ is odd, this assumption is already included the conditions of Lemma \ref{Claim 3}.} Then for all $\ell\geq f_0(\Sigma)$, there exists a $k$-neighborly PL $2k$-sphere $\Sigma_\ell$ with $\ell$ vertices,  all of whose missing faces have dimension $k$.  Furthermore, if $\Sigma$ is polytopal, then all spheres $\Sigma_\ell$ produced by this construction are also polytopal.
\end{corollary}
\proof The first part follows by starting with $(\Sigma,E)$ and inductively applying Lemma \ref{Claim 4}. (In the case that $k$ is even, the fact that all missing faces of resulting spheres have dimension $k$ follows from Lemma \ref{Claim 1}.) 

To prove the polytopality part, observe that we are sewing on the ball $D_k$, and that by definition, $D_k$ can be expressed as
$\st(F_1)\backslash \left(\st(F_2)\backslash \big( \dots\backslash \big(\st(F_{k-1})\backslash \st(F_k)\big)\dots\big)\right)$, where $F_j=e_1\cup \dots \cup e_j$ in the initial step (see Lemma \ref{Claim 3}) and $F_j=e'_1\cup \dots \cup e'_j$ , with $e_i'$ defined as in Lemma \ref{Claim 4} in the inductive steps. The polytopality statement then follows from a result of Shemer \cite[Lemma 4.4]{Shemer}.
\endproof

We are ready to prove Theorem \ref{thm: neighborly 2k-sphere m_k+1=0}.

\smallskip\noindent {\it Proof of Theorem \ref{thm: neighborly 2k-sphere m_k+1=0}: \ }
In the case of $k=2$, take $\Sigma=\partial \mathcal{P}^0_{42}$, the boundary complex of the $9$-vertex neighborly $5$-polytope $\mathcal{P}^0_{42}$ (see Section 7.2). All missing faces of $\Sigma$ have dimension $2$. One can easily check from (\ref{list of facets}) that $\lk(\{1,9\})$ in $\partial\mathcal{P}^0_{42}$ is a flag sphere with vertex set $\{2,3,\dots,8\}$. Hence $\Sigma$ and $E=\{e_1=\{1,9\}, e_2=\{3,6\}\in \lk(\{1,9\})\}$ satisfy the conditions of Lemma \ref{Claim 3}. The statement of the theorem then follows from Corollary \ref{main-cor}.

Assume now that $k=2i-1\geq 3$ is odd. Take $\Sigma$ to be the boundary complex of the polytope $Q_k$ from Definition \ref{Q_i}. By Proposition \ref{prop:neighborliness-of-Q_i},  $\partial Q_k$ is a $k$-neighborly $2k$-sphere with $2k+4$ vertices all of whose missing faces have dimension $k$. 
Corollary \ref{main-cor} and Lemma \ref{lem:links-of-Q_i} then imply the statement.
\hfill$\square$\medskip

In view of Theorem \ref{thm: neighborly 2k-sphere m_k+1=0} and Proposition \ref{Prop: vertex-minimal}, it is natural to ask the following (cf.~Question \ref{question:m_{k+1}=0}).
\begin{question}
	Let $k\geq 4$ be even. Is there an infinite family of neighborly $2k$-spheres (or neighborly $(2k+1)$-polytopes) with arbitrary number $n\geq 2k+5$ of vertices, all of whose missing faces have dimension $k$? 
\end{question}

The smallest $k$ for which we do not know if an infinite family of neighborly $2k$-spheres all of whose missing faces have dimension $k$ exists is $k=4$. The complex $8\_14\_1\_1$ from the Manifold page \cite{manifold_page} is a neighborly $8$-sphere with $14$ vertices all of whose missing faces have dimension $4$. Unfortunately, this sphere has no sequence $ \{e_1, e_2, e_3, e_4\}$ of edges that satisfies conditions of Lemma \ref{Claim 3}, and so our inductive procedure does not apply. We also do not know if there exists a neighborly $8$-sphere with fewer than $14$ vertices and $m_5=0$.

We close the paper with one additional problem. Let $d\geq 4$, let $1\leq k\leq \lfloor d/2\rfloor -1$, and let $\Delta$ be a simplicial $(d-1)$-sphere with $g_k\neq 0$. If $g_{k+1}=0$, then $\Delta$ is $k$-stacked, and so by Theorem \ref{lem:upper-bounds}, $m_{d-k}=g_k\neq 0$. On the other extreme is the case where $(g_{k+1})_{\langle k+1\rangle}=g_k$, or equivalently, $(g_k-1)^{\langle k \rangle}<g_{k+1}\leq g_k^{\langle k \rangle}$. In this case, by Theorem \ref{lem:upper-bounds}, $m_{d-k}=0$. This discussion, along with Theorem \ref{thm: neighborly 2k-sphere m_k+1=0} and Corollary~\ref{cor:characterization of m for 2-spheres}, motivates the following problem.

\begin{problem}
Let $g=(g_1,\ldots, g_{\lfloor d/2\rfloor})$ be an integer vector that satisfies all the conditions of the $g$-theorem and all of whose entries are strictly positive. What are the (additional) necessary and sufficient conditions on $g$ for the existence of a simplicial $(d-1)$-sphere $\Delta$ such that $g(\Delta)=g$ and $m_{d-i}(\Delta)=0$ for all $i\leq  \lceil d/2 \rceil -1$?
\end{problem}

\section*{Acknowledgments}
We are grateful to Gil Kalai for bringing to our attention recent results on the clique density problem, to Luz Elena Grisales 
G\'omez for carefully reading previous versions and spotting several mistakes, and to the anonymous referee for comments that led to Remark \ref{rem:Delta_i}.

{\small
	\bibliography{refs}
	\bibliographystyle{plain}
}
\end{document}